\input ./style/arxiv-general.cfg
\documentclass[aop,MSNbibl,dvips]{arximspdf}
\makeatletter
   \@ifpackageloaded{graphicx}{}{\usepackage{graphicx}}
\makeatother
\usepackage{mathrsfs}

\doi{10.1214/14-AOP926} 
\volume{43}
\issue{4}
\pubyear{2015}
\firstpage{1992}
\lastpage{2025}

\makeatletter
\newcommand{\rrvert}{\vert}
\newcommand{\llvert}{\vert}
\def\N{\mathbf{N}}
\def\Z{\mathbf{Z}}
\def\R{\mathbf{R}}
\def\Hb{{\mathcal H}}
\def\Cc{{\mathcal C}}
\def\Ec{{\mathcal E}}
\def\Bc{{\mathscr B}}
\def\ebf{\mathbf{e}}
\def\tbar{\bar{t}}
\def\tlarge{\bar{t}}
\def\muZ{\mu}
\def\Cyl{\Cc}
\def\Zdiv{\mathscr{Z}}
\def\Tdiv{\mathscr{T}}
\def\Pr{\mathbf{P}}
\newcommand{\E}{\mathbf{E}}
\newcommand{\Var}{\operatorname{Var}}
\newcommand{\ind}{\mathbf{1}}
\newcommand{\diam}{\operatorname{diam}}
\newtheorem{teo}{Theorem}
\newtheorem{lma}[teo]{Lemma}
\newtheorem{cor}[teo]{Corollary}
\newtheorem{prop}[teo]{Proposition}
\newproclaim{remark}{Remark}
\makeatother

\begin{document}
\begin{frontmatter}

\title{A Hsu--Robbins--Erd\H{o}s strong law\\ in first-passage percolation}
\runtitle{A Hsu--Robbins--Erd\H{o}s strong law in FPP}

\begin{aug}
\author[A]{\fnms{Daniel} \snm{Ahlberg}\corref{}\ead[label=e1]{ahlberg@impa.br}}
\runauthor{D. Ahlberg}
\affiliation{Instituto Nacional de Matem\'atica Pura e Aplicada}
\address[A]{Instituto Nacional de Matem\'atica Pura e Aplicada\\ 
Estrada Dona Castorina 110\\
22460-320 Rio de Janeiro\\
Brasil\\
\printead{e1}} 
\end{aug}

\received{\smonth{6} \syear{2013}}
\revised{\smonth{1} \syear{2014}}

%
\begin{abstract}
Large deviations in the context of first-passage percolation was first
studied in the early 1980s
by Grimmett and Kesten, and has since been revisited in a variety of
studies. However,
none of these studies provides a \mbox{precise} relation between the
existence of moments of
polynomial order and the decay of probability tails. Such a relation is
derived in this
paper, and is used to strengthen the conclusion of the shape theorem.
In contrast to
its one-dimensional counterpart---the Hsu--Robbins--Erd\H{o}s strong law---this
strengthening is obtained without imposing a higher-order moment condition.
\end{abstract}

%
\begin{keyword}[class=AMS]
\kwd[Primary ]{60K35}
\kwd[; secondary ]{60F10}
\kwd{60F15}
\end{keyword}
\begin{keyword}
\kwd{First-passage percolation}
\kwd{shape theorem}
\kwd{large deviations}
\end{keyword}

\end{frontmatter}

\section{Introduction}

The study of large deviations in first-passage percolation was
pioneered by Grimmett and Kesten~\cite{grikes84}. In their work, they
investigate the rate of convergence of travel times toward the
so-called time constant, by providing some necessary and sufficient
conditions for exponential decay of the probability of linear order
deviations. Although the rate of convergence toward the time constant
has received considerable attention in the literature, there is no
systematic study of the regime for polynomial decay of the probability
tails. This is remarkable since it is precisely in this regime that
strong laws such as the celebrated shape theorem are obtained. In this
paper, we derive a precise characterization of the regime of polynomial
decay in terms of a moment condition. As a consequence, we improve upon
the statement of the shape theorem without strengthening its hypothesis.

Consider the $\Z^d$ nearest-neighbor lattice for $d\ge2$, with
nonnegative i.i.d. random weights assigned to its edges. The random
weights induce a random pseudo-metric on $\Z^d$, known as
first-passage percolation, in which distance between points are given
by the minimal weight sum among possible paths. An important
subadditive nature of these distances, sometimes referred to as \emph
{travel times}, was in the 1960s identified and studied by Hammersley
and Welsh~\cite{hamwel65} and Kingman~\cite{kingman68}. A particular
fact dating back to these early studies is that, under weak conditions,
travel times grow linearly with respect to Euclidean distance between
points, at a rate depending on the direction, and with sublinear
corrections. This asymptotic rate is referred to as the \emph{time constant}.


The work of Grimmett and Kesten on the rate of convergence toward the
time constant was later continued by Kesten himself in~\cite
{kesten86}, and more recently refined in~\cite{chozha03} and~\cite
{cragaumou09}. These studies, just like the present study, investigates
deviations of linear order. Other studies have pursued stronger
concentration inequalities, describing deviations of sublinear order,
notably Kesten~\cite{kesten93}, Talagrand~\cite{talagrand95} and
Benjamini, Kalai and Schramm~\cite{benkalsch03}.
Other authors have considered deviations of linear order for the
related concept of so-called chemical distance in Bernoulli percolation
\cite{antpis96,garmar07}. A common feature among these studies is that
they aim to derive exponential decay of the probability tails for
deviations in coordinate directions, and generally require exponential
decay of the tails of the weights in order to get there. There is no
previous study characterizing the regime of polynomial rate of decay on
the probability tails of linear order deviations. This study attends to
this matter and provides necessary and sufficient conditions for
polynomial rate of decay in terms of a moment condition, valid for all
directions simultaneously. The results obtained strengthens earlier
strong laws in first-passage percolation, in particular the shape
theorem, due to Richardson~\cite{richardson73} and Cox and
Durrett~\cite{coxdur81}, which states the precise conditions under
which the set of points in $\Z^d$ within distance $t$ from
the origin (in the random metric), rescaled by $t$, converges to a
deterministic compact and convex set.

\subsection{The shape theorem}

We will throughout this paper assume that $d\ge2$, as the
one-dimensional case coincides with the study of i.i.d. sequences. Let
$\Ec$ denote the set of edges of the $\Z^d$ lattice, and let $\tau
_e$ denote the random weight associated with the edge $e$ in $\Ec$.
The collection $\{\tau_e\}_{e\in\Ec}$ of weights, commonly referred
to as \emph{passage times}, will throughout be assumed to be
nonnegative and i.i.d. The distance, or \emph{travel time}, $T(y,z)$
between two points $y$ and $z$ of $\Z^d$ is defined as the minimal
path weight, as induced by the random environment $\{\tau_e\}_{e\in
\Ec}$, among paths connecting $y$ and $z$. That is, given a path
$\Gamma$, let $T(\Gamma):=\sum_{e\in\Gamma}\tau_e$ and define
\[
T(y,z):=\inf\bigl\{T(\Gamma)\dvtx \Gamma\mbox{ is a path connecting $y$ and $z$}
\bigr\}.
\]
As mentioned above, travel times grow linearly in comparison with
Euclidean or $\ell^1$-distance on $\Z^d$, denoted below by $|\cdot|$
and $\|\cdot\|$, respectively. The precise meaning of this informal
statement refers to the existence of the limit
%
\begin{equation}
\label{deftimeconstant} \muZ(z):=\lim_{n\to\infty}\frac
{T(0,nz)}{n}\qquad
\mbox{in probability},
\end{equation}
which we refer to as the \emph{time constant}, and which may depend on
the direction. Indeed, the growth is only linear in the case that $\muZ
(z)>0$, which is known to be the case if and only if $\Pr(\tau
_e=0)<p_c(d)$, where $p_c(d)$ denotes the critical probability for bond
percolation on $\Z^d$ (see~\cite{kesten86}).

Existence of the limit in~(\ref{deftimeconstant}) was first obtained
with a moment condition in~\cite{hamwel65}, but indeed exists finitely
without any restriction on the passage time distribution (see \cite
{coxdur81,kesten86}, and the discussion in Appendix~\ref{appmu}
below). Moreover, the convergence in~(\ref{deftimeconstant}) holds
almost surely and in $L^1$ if and only if $\E[Y]<\infty$, where $Y$
denotes the minimum of $2d$ random variables distributed as $\tau_e$,
as a consequence of the subadditive ergodic theorem~\cite{kingman68}.
A more comprehensive result, the shape theorem, provides simultaneous
convergence in all directions. A weak form thereof can be concisely
stated as
%
\begin{equation}
\label{eqweakST} \limsup_{z\in\Z^d\dvtx  \|z\|\to\infty}\Pr\bigl
(\bigl|T(0,z)-\muZ(z)\bigr|>
\varepsilon\|z\| \bigr)=0\qquad\mbox{for every }\varepsilon>0.
\end{equation}
Also the convergence in~(\ref{eqweakST}) holds without restrictions
to the passage time distribution (see Section~\ref{secshell}).
However, under the assumption of a moment condition, it is possible to
obtain an estimate on the rate of decay in~(\ref{eqweakST}). This is
precisely the aim of this study, and the content of Theorems~\ref
{teoLDbelow} and~\ref{teoLDabove} below. Based thereon, we will
derive the following Hsu--Robbins--Erd\H{o}s type of strong law, which
characterizes the summability of tail probabilities of the above type.

\begin{teo}\label{teoHRE}
For every $\alpha>0$, $\varepsilon>0$ and $d\ge2$,
\[
\E\bigl[Y^\alpha\bigr]<\infty\quad\Longleftrightarrow\quad\sum
_{z\in\Z^d}\| z\|^{\alpha-d} \Pr\bigl(\bigl|T(0,z)-\muZ(z)\bigr|>
\varepsilon\|z\| \bigr)<\infty.
\]
\end{teo}

Apart from characterizing the summability of probabilities of large
deviations away from the time constant, Theorem~\ref{teoHRE} has
several implications for the shape theorem, which we will discuss next.
Cox and Durrett's version of the shape theorem is a strengthening
of~(\ref{eqweakST}), and can be stated as follows: if $\E
[Y^d]<\infty$, then
%
\begin{equation}
\label{eqshapelimit} \limsup_{z\in\Z^d\dvtx  \|z\|\to\infty}\frac{
|T(0,z)-\muZ
(z) |}{\|z\|}=0\qquad
\mbox{almost surely}.
\end{equation}
(A more popular way to phrase the shape theorem is offered below.) It
is well known that $\E[Y^d]<\infty$ is also necessary for the
convergence in~(\ref{eqshapelimit}). Let
\[
\Zdiv_\varepsilon:= \bigl\{z\in\Z^d\dvtx \bigl|T(0,z)-\muZ(z)\bigr|>
\varepsilon\|z\| \bigr\},
\]
and note that~(\ref{eqshapelimit}) is equivalent to saying that the
cardinality of the set $\Zdiv_\varepsilon$, denoted by $|\Zdiv
_\varepsilon|$, is finite for every $\varepsilon>0$ with
probability one. This statement is by Theorem~\ref{teoHRE}, under the
same assumption of $\E[Y^d]<\infty$, strengthened to say that $\E
|\Zdiv_\varepsilon|<\infty$ for every $\varepsilon>0$.

The shape theorem is commonly presented as a comparison between the
random set
\[
\Bc_t:= \bigl\{z\in\Z^d\dvtx T(0,z)\le t \bigr\}
\]
and the discrete ``ball'' $\Bc^{\muZ}_t:=\{z\in\Z^d\dvtx \muZ(z)\le t\}$.
The growth of the first-passage process may be divided into two regimes
characterized by the time constant: \mbox{$\muZ\equiv0$}, and $\muZ(z)>0$
for all $z\neq0$. In the more interesting regime $\muZ\not\equiv0$,
Cox and Durrett's shape theorem can be phrased: if $\E[Y^d]<\infty$,
then for every $\varepsilon>0$ the two inclusions
%
\begin{equation}
\label{eqshapeinclusion} \Bc^{\muZ}_{(1-\varepsilon)t} \subset\Bc_t
\subset\Bc^{\muZ}_{(1+\varepsilon)t}
\end{equation}
hold for all $t$ large enough, with probability one.
A simple inversion argument shows that this formulation is equivalent
to the one in~(\ref{eqshapelimit}).

The time constant $\muZ(\cdot)$ extends continuously to $\R^d$ and
inherits the properties of a semi-norm. In other words, one may
interpret~(\ref{eqshapeinclusion}) as $\frac{1}{t}\Bc_t$ being
asymptotic to the unit ball $\{x\in\R^d\dvtx \muZ(x)\le1\}$ expressed in
this norm, and failure of either inclusion in~(\ref
{eqshapeinclusion}) indicates a linear order deviation of $\Bc_t$
from this asymptotic shape. Inspired by Theorem~\ref{teoHRE}, one may
wonder whether the size, that is Lebesgue measure, of the set of times
for which~(\ref{eqshapeinclusion}) fails behaves similarly as the
size of $\Zdiv_\varepsilon$. Assume that $\muZ\not\equiv0$ and let
\[
\Tdiv_\varepsilon:= \bigl\{t\ge0\dvtx  \mbox{either inclusion in
(\ref
{eqshapeinclusion}) fails} \bigr\}.
\]
The shape theorem says that $\E[Y^d]<\infty$ is sufficient for the
supremum of $\Tdiv_\varepsilon$, and hence the Lebesgue measure
$|\Tdiv_\varepsilon|$ of the set $\Tdiv_\varepsilon$, to be
finite almost surely, for every $\varepsilon>0$. As it turns out, the
same condition is not sufficient to obtain finite expectation, as our
next result shows.

\begin{teo}\label{teoST}
Assume that $\muZ\not\equiv0$, and let $\alpha>0$, $\varepsilon
>0$ and $d\ge2$. Then
%
\[
\E\bigl[Y^{d+\alpha}\bigr]<\infty\quad\Longleftrightarrow\quad\E\bigl
[|\Tdiv
_\varepsilon|^\alpha\bigr]<\infty\quad\Longleftrightarrow\quad\E
\bigl[(\sup\Tdiv_\varepsilon)^\alpha\bigr]<\infty.
\]
\end{teo}

The title of the paper refers to Theorem~\ref{teoHRE}, and is
motivated by the following comparison with its one-dimensional
analogue. Let $S_n$ denote the sum of $n$ i.i.d. random variables with
mean $m$. The strong law of large numbers states that $S_n/n$ converges
almost surely to its mean as $n$ tends to infinity. This is the
one-dimensional analogue of the shape theorem. Equivalently put, the
number of $n$ for which $|S_n-nm|>\varepsilon n$ is almost surely
finite for every $\varepsilon>0$. This is true if and only if the
mean $m$ is finite. Hsu and Robbins~\cite{hsurob47} proved that if
also second moments are finite, then
\[
\sum_{n=1}^\infty
\Pr\bigl(|S_n-nm|>\varepsilon n\bigr) < \infty.
\]
Erd\H{o}s \cite{erdos49,erdos50} showed that finite second moment is
in fact also necessary for this stronger conclusion to hold. In
particular, the stronger conclusion requires a stronger hypothesis. The
analogous strengthening of the shape theorem (Theorem~\ref{teoHRE})
holds without the need of a stronger hypothesis.

\begin{remark*}
A sequence $X_1,X_2,\ldots$ of random variables which for all
$\varepsilon>0$ and some random variable $X$ satisfies $\sum
_{n=1}^\infty\Pr(|X_n-X|>\varepsilon)<\infty$ is necessarily
convergent to $X$ almost surely. As introduced by Hsu and Robbins, this
stronger mode of convergence was, ``for want of a better name,'' by them
called \emph{complete}. In their language, Theorem~\ref{teoHRE}
implies that $\E[Y^d]<\infty$ is necessary and sufficient not only
for the almost sure, but also for complete convergence in the shape theorem.
\end{remark*}

\begin{remark*}
In the\vspace*{2pt} statement of the shape theorem,~(\ref{eqshapeinclusion}) is
often exchanged for $(1-\varepsilon)\widetilde\Bc^{\muZ}\subset\frac
{1}{t}\widetilde\Bc_t\subset(1+\varepsilon)\widetilde\Bc^{\muZ}$, where
$\widetilde\Bc_t$ is the ``fattened'' set obtained by replacing each site
in $\Bc_t$ by a unit cube centered around it, and where $\widetilde\Bc
^{\muZ}=\{x\in\R^d\dvtx \muZ(x)\le1\}$. This formulation is equivalent
to the one based on~(\ref{eqshapeinclusion}). In Section~\ref
{sectimediv}, where Theorem~\ref{teoST} is proved, it will be clear
why it is more convenient to work with the discrete sets in~(\ref
{eqshapeinclusion}).
\end{remark*}

\subsection{Large deviation estimates}

Grimmett and Kesten were concerned with large deviation estimates and
large deviation principles for travel times in coordinate directions,
such as the family $\{T(0,n\ebf_1)-n\muZ(\ebf_1)\}_{n\ge1}$. Their
results in \cite{grikes84} were soon improved upon in \cite
{kesten86}. Deviations above and below the time constants behave quite
differently. The first observation in this direction is that the
probability of large deviations below the time constant decay at an
exponential rate without restrictions to the passage time distribution.
This was proved for travel times in coordinate direction in \cite{grikes84}, $d=2$, and \cite{kesten86}, $d\ge2$. A further indication
is that the probability of deviations above the time constant decays
superexponentially, subject to a sufficiently strong moment condition
on $Y$ (at least exponential). This observation was first made by
Kesten~\cite{kesten86}, Theorem~5.9, but more recently refined by Chow
and Zhang~\cite{chozha03} and Cranston, Gauthier and Mountford~\cite{cragaumou09}.

The main contribution of this study is an estimate on linear order
deviations above the time constant under the assumption that $\E
[Y^\alpha]<\infty$ for some $\alpha>0$, presented in Theorem~\ref
{teoLDabove} below. The proof of this result will make crucial use of
the fact that what determines the rate of decay of the probability tail
of the travel time $T(0,z)$ is the distribution of the weights of the
$2d$ edges reaching out of the origin and the $2d$ edges leading in to
the point $z$. This idea is in itself not new. A similar idea was used
by Cox and Durrett to prove existence of the limit in~(\ref
{deftimeconstant}) without any moment condition (see~\cite{coxdur81}
for $d=2$, and~\cite{kesten86} for higher dimensions). Also Zhang
\cite{zhang10} builds on their work and obtains concentration
inequalities based on a moment condition for the travel time between
sets whose size is growing logarithmically in their distance.

We provide two estimates for deviations of linear order, considering
deviations below and above the time constant separately. The first
extends the result of Grimmett and Kesten~\cite{grikes84,kesten86}
from the coordinate axis to all of $\Z^d$.

\begin{teo}\label{teoLDbelow}
For every $\varepsilon>0$, there are $M=M(\varepsilon)$ and $\gamma
=\gamma(\varepsilon)>0$ such that for every $z\in\Z^d$ and $x\ge\|
z\|$
\[
\Pr\bigl(T(0,z)-\muZ(z)<-\varepsilon x \bigr) \le M e^{-\gamma x}.
\]
\end{teo}

A similar exponential rate of decay cannot hold in general for
deviations above the time constant, since $\Pr(T(0,z)-\muZ
(z)>\varepsilon x)$ is bounded from below by $\Pr(Y>Mx)$ for any
sufficiently large $M$. One may instead wonder whether the decay of
$\Pr(T(0,z)-\muZ(z)>\varepsilon x)$ in fact is determined by the
probability tails of~$Y$. Interestingly, in the regime of polynomial
decay on the probability tails, this is indeed the case.

\begin{teo}\label{teoLDabove}
Assume that $\E[Y^\alpha]<\infty$ for some $\alpha>0$. For every
$\varepsilon>0$ and $q\ge1$ there exists $M=M(\alpha,\varepsilon,q)$
such that for every $z\in\Z^d$ and $x\ge\|z\|$
\[
\Pr\bigl(T(0,z)-\muZ(z)>\varepsilon x \bigr) \le M \Pr(Y>x/M)+
\frac{M}{x^q}.
\]
\end{teo}

The strength in Theorem~\ref{teoLDbelow} is that it gives exponential
decay without the need of a moment condition, which is best possible;
see~\cite{kesten86}, Theorem~5.2. Moreover, the upper bound is
independent of the direction. This fact is a consequence of the
equivalence between $\muZ$ and the usual $\ell^p$-distances.

The strength in Theorem~\ref{teoLDabove} is that under a minimal
moment assumption, it relates the probability tail of $T(0,z)-\muZ(z)$
directly with that of $Y$, together with an additional error term. In
the case of polynomial decay of the tails of $Y$, this result is
essentially sharp. It is not clear whether the polynomially decaying
error term that appears in Theorem~\ref{teoLDabove} could be improved
or not. However, in view of the exponential decay obtained in the case
of a moment condition of exponential order, and the superexponential
decay for bounded passage times (see \cite{kesten86}, Theorem~5.9, or
\cite{chozha03} and \cite{cragaumou09}), it seems possible that the
error in fact may decay at least exponentially fast. This is the most
interesting question left open in this study, together with the
question whether it is possible to remove the moment condition in
Theorem~\ref{teoLDabove} completely.

Theorem~\ref{teoHRE} is easily derived from Theorem~\ref
{teoLDbelow} and Theorem~\ref{teoLDabove}. The proof of Theorem~\ref
{teoLDbelow} will follow the steps of \cite{kesten86}, whereas the
proof of Theorem~\ref{teoLDabove} will be derived from first
principles, via a regeneration argument similar to that used in~\cite
{A-1}. A similar characterization of deviations away from the time
constant as the one presented here has in parallel been derived for
first-passage percolation of cone-like subgraphs of the $\Z^d$ lattice
by the same author in~\cite{A-3}, however, detailed proofs appear only here.
Complementary to Theorem~\ref{teoHRE}, we may also obtain necessary
and sufficient conditions for summability of tails in radial directions
from Theorems~\ref{teoLDbelow} and~\ref{teoLDabove}.

\begin{cor}\label{corLDsum}
For any $\alpha>0$, $\varepsilon>0$ and $z\in\Z^d$,
\[
\E\bigl[Y^\alpha\bigr]<\infty\quad\Longleftrightarrow\quad\sum
_{n=1}^\infty n^{\alpha-1} \Pr\bigl(\bigl|T(0,nz)-n
\muZ(z)\bigr|>\varepsilon n \bigr)<\infty.
\]
\end{cor}

Another consequence of Theorems~\ref{teoLDbelow} and~\ref
{teoLDabove} is the following characterization of $L^p$-convergence,
of which a proof may be found in~\cite{A-3}.

\begin{cor}\label{corLp}
For every $p>0$,
\[
\E\bigl[Y^p\bigr]<\infty\quad\Longleftrightarrow\quad\limsup
_{z\in\Z^d\dvtx  \|z\|
\to\infty}\E\biggl\llvert\frac{T(0,z)-\muZ(z)}{\|z\|}\biggr\rrvert
^p=0.
\]
\end{cor}

Constants given above and also later on in this paper generally depend
on the dimension $d$ and on the actual passage time distribution.
However, this will not always be stressed in the notation. We would
also like to remind the reader that above and for the rest of this
paper we will let $|\cdot|$ denote Euclidean distance, and let $\|
\cdot\|$ denote $\ell^1$-distance. Although the former notation will
also be used to denote cardinality for discrete sets, and Lebesgue
measure for (measurable) subsets of $\R^d$, we believe that what is
referred to will always be clear form the context. Finally, we will
denote the $d$ coordinate directions by $\ebf_i$ for $i=1,2,\ldots,d$,
and recall that $Y$ denotes the minimum of $2d$ independent random
variables distributed as $\tau_e$.

We continue this paper with a discussion of some preliminary results
and observations in Section~\ref{secprel}. In Section~\ref
{secLDbelow}, we prove Theorem~\ref{teoLDbelow}, and in Section~\ref
{secreg} we describe a regenerative approach that in Section~\ref
{secLDabove} will be used to prove Theorem~\ref{teoLDabove}.
Finally, Theorem~\ref{teoHRE} is derived in Section~\ref
{secHREproof}, and Theorem~\ref{teoST} in the ending Section~\ref
{sectimediv}.

\section{Convergence toward the asymptotic shape}\label{secprel}

Before moving on to the core of this paper, we will first discuss some
preliminary observations and results for later reference. We will begin
with a few properties of the time constant, and their consequences for
the asymptotic shape $\{x\in\R^d\dvtx \muZ(x)\le1\}$.
We thereafter describe Cox, Durrett and Kesten's approach to
convergence without moment condition, in order to state Kesten's
version of the shape theorem. Kesten's theorem will be required in
order to prove Theorem~\ref{teoLDbelow} without moment condition; a
first application is found in Proposition~\ref{propTshape} below.

\subsection{The time constant and asymptotic shape}

The foremost characteristic of first-passage percolation is its
subadditive property, inherited from its interpretation as a
pseudo-metric on $\Z^d$. This property takes the expression
\[
T(x,z) \le T(x,y)+T(y,z)\qquad\mbox{for all }x,y,z\in\Z^d,
\]
and will be used repeatedly throughout this study. Subadditivity also
carries over in the limit. The time constant $\muZ$ was defined
in~(\ref{deftimeconstant}) on $\Z^d$, but extends in fact
continuously to all of $\R^d$.
The extension is unique with respect to preservation of the following
properties:
\begin{eqnarray*}
\muZ(ax)&=& |a|\muZ(x)\qquad\mbox{for }a\in\R\mbox{ and }x\in\R^d,
\\
\muZ(x+y)&\le& \muZ(x)+\muZ(y)\qquad\mbox{for }x,y\in\R^d,
\\
\bigl|\muZ(x)-\muZ(y)\bigr| &\le&\muZ(\ebf_1)\|x-y\|\qquad\mbox{for }x,y\in
\R^d.
\end{eqnarray*}
The third of the above properties is easily obtained from the previous
two, and shows that $\muZ\dvtx \R^d\to[0,\infty)$ is Lipschitz continuous.

As mentioned above, there are two regimes separating the behavior of
$\muZ$. Either $\muZ\equiv0$, or $\muZ(x)\neq0$ for all $x\neq0$.
The separating factor is, as mentioned, whether $\Pr(\tau_e=0)\ge
p_c(d)$ or not, where $p_c(d)$ denotes the critical probability for
bond percolation on $\Z^d$. In the latter regime $\muZ$ satisfies all
the properties of a norm on~$\R^d$, and the unit ball $\{x\in\R
^d\dvtx \muZ(x)\le1\}$ in this norm can be shown to be compact convex and
to have nonempty interior. Consequently, $\mu$ is bounded away from 0
and infinity on any compact set not including the origin. In particular,
\[
0 < \inf_{\|x\|=1}\muZ(x) \le\sup_{\|x\|=1}\muZ(x)
< \infty.
\]
A careful account for the above statements is found in~\cite
{kesten86}. (See also Appendix~\ref{appmu} below.)

\subsection{A shape theorem without moment condition}\label{secshell}

Cox and Durrett~\cite{coxdur81} found a way to prove existence of the
limit in~(\ref{deftimeconstant}) without restrictions to the passage
time distribution. Their argument was presented for $d=2$, and later
extended to higher dimensions by Kesten~\cite{kesten86}. As a
consequence, Kesten showed that the moment condition in the shape
theorem can be removed to the cost of a weakening of its conclusion.
Since Kesten's result will be important in order to derive an estimate
on large deviations below the time constant (Theorem~\ref
{teoLDbelow}), we will recall the result here. To reproduce the result
in a fair amount of detail requires that some notation is introduced.
However, a bit loosely put, Kesten's result states that if $\Bc_t$ is
replaced by the set $\overline{\Bc}_t$ containing $\Bc_t$ and each other
point in $\Z^d$ ``surrounded'' by~$\Bc_t$, then~(\ref
{eqshapeinclusion}) holds for all large enough $t$ almost surely also
without the moment condition. That is, $\overline{\Bc}_t$ should be
thought of as containing
all points from which there is no infinite self-avoiding path disjoint
with $\Bc_t$.

Given $\delta>0$, pick $\tlarge=\tlarge(\delta)$ such that $\Pr
(\tau_e\le\tlarge)\ge1-\delta$. Next, color each vertex in $\Z^d$
either black or white depending on whether at least one of the edges
adjacent to it has weight larger than $\tlarge$ or not. The moral here
is that if $\delta$ is small, then an infinite connected component of
white vertices will exist with probability one, and that travels within
this white component are never ``slow.'' Based on this idea, we go on and
define ``shells'' of white vertices around each point in $\Z^d$. If we
make sure that these shells exist, are not too large, but intersect the
infinite white component, then the shells may be used to ``surround''
points in $\Z^d$.

Without\vspace*{1pt} reproducing all the details, Kesten shows that it is possible
to define a set $\Delta_z\subset\Z^d$ consisting of white vertices
and which, given that $\delta=\delta(d)>0$ is sufficiently small,
almost surely satisfies the following properties (see Appendix~\ref
{appshell} below):
\begin{longlist}[(3)]
\item[(1)] $\Delta_z$ is a finite connected subset of $\Z^d$,
\item[(2)] every path connecting $z$ to infinity has to intersect $\Delta_z$,
\item[(3)] there is a point in $\Delta_z$ which is connected to infinity by
a path of white vertices,
\item[(4)] either every path between $y$ and $z$ in $\Z^d$ intersects both
$\Delta_y$ and $\Delta_z$, or $\Delta_y\cap\Delta_z\neq\varnothing$.
\end{longlist}
Moreover, the shells may be chosen so that their diameter, defined as
the maximal $\ell^1$-distance between a pair of its elements, for some
$M<\infty$ and $\gamma>0$ satisfies
%
\begin{equation}
\label{eqshelldecay} \Pr\bigl(\diam(\Delta_z)>n\bigr) \le
Me^{-\gamma n}\qquad\mbox{for all }n\ge1.
\end{equation}

The advantage of the construction of shells is that although travel
times between points may be too heavy-tailed to obey a strong law, the
travel time between shells of two points in $\Z^d$ have finite moments
of all orders. That $T(\Delta_y,\Delta_z)$ is a lower bound for
$T(y,z)$ is a consequence of the fourth property. A complementary upper
bound is obtained by summing the weights of paths connecting $y$ and
$\Delta_y$, $\Delta_y$ and~$\Delta_z$, and $\Delta_z$ and $z$,
respectively. These paths may not intersect and form a path between $y$
and $z$, so in order to obtain an upper bound, we also have to consider
the maximal weight of a path between two points in $\Delta_y$ and
$\Delta_z$, respectively. Since each shell is white and connected, and
the $2d$ edges adjacent to a white vertex have weight at most $\tlarge
$, we arrive at the following inequality:
%
\begin{equation}
\label{eqshellcompare} 0\le T(y,z)-T(\Delta_y,\Delta_z)\le T(y,
\Delta_y)+T(\Delta_z,z)+2\,d\tlarge\bigl(|
\Delta_y|+|\Delta_z|\bigr).
\end{equation}
Without the need of a moment condition $(T(\Delta_0,\Delta
_{nz})+2\,d\tlarge|\Delta_{nz}|)_{n\ge1}$ is found to satisfy the
conditions of the subadditive ergodic theorem; see~\cite{kesten86}, Theorem~2.26. Consequently, the limit of $\frac
{1}{n}T(\Delta_0,\Delta_{nz})$ as $n\to\infty$ exists almost surely
and in $L^1$, and together with~(\ref{eqshellcompare}), existence of
the limit in~(\ref{deftimeconstant}) is obtained.

Let us now move on to state Kesten's version of the shape theorem. For
our purposes, it will be practical to present the statement on the form
of a limit, in analogy to~(\ref{eqshapelimit}). On this form,
Kesten's theorem \cite{kesten86}, Theorem~3.1, simply states that
%
\begin{equation}
\label{eqSTkesten} \limsup_{z\in\Z^d\dvtx  \|z\|\to\infty}\frac{|T(0,\Delta
_z)-\muZ
(z)|}{\|z\|}=0\qquad
\mbox{almost surely}.
\end{equation}
The weak version of the shape theorem stated in~(\ref{eqweakST}) is
now easily obtained from~(\ref{eqSTkesten}) together with~(\ref
{eqshellcompare}). As a comparison, recall that Cox and Durrett's
version of the shape theorem states that if $\E[Y^d]<\infty$,
then~(\ref{eqSTkesten}) holds also if $\Delta_z$ is replaced by $z$.

\subsection{Point-to-shape travel times}

In view of the convergence of the set $\Bc_t$ toward a convex compact
set described in terms of $\muZ(\cdot)$, it is reasonable to study
the travel time to points at a large distance with respect to this
norm. That is, introduce what could be referred to as \emph
{point-to-shape} travel times as $T(0,\neg\Bc^{\muZ}_n)$, where
$\neg\Bc^{\muZ}_n:=\Z^d\setminus\Bc^{\muZ}_n=\{z\in\Z^d\dvtx \muZ
(z)>n\}$. This definition only makes sense in the case that $\muZ\not
\equiv0$, and in this case a strong law for the point-to-shape travel
times holds without restriction on the passage time distribution.

\begin{prop}\label{propTshape}
Assume that $\muZ\not\equiv0$. Then
\[
\lim_{n\to\infty}\frac{T(0,\neg\Bc^{\muZ}_n)}{n}=1\qquad\mbox{almost surely}.
\]
\end{prop}

\begin{pf}
Let $m_n$ denote the least integer for which $m_n\muZ(\ebf_1)>n$. By
definition, we have
\[
\frac{T(0,\neg\Bc^{\muZ}_n)}{n} \le\frac{T(0,\Delta_{m_n\ebf
_1})+\tlarge|\Delta_{m_n\ebf_1}|}{n} \to1\qquad\mbox{almost surely}.
\]
So, it is sufficient to show that the event
\[
A_\delta= \biggl\{\liminf_{n\to\infty}\frac{T(0,\neg\Bc^{\muZ
}_n)}{n}
\le1-\delta\biggr\}
\]
has probability 0 to occur for every $\delta>0$. On the event
$A_\delta$ there is an increasing sequence $(n_k)_{k\ge1}$ of
integers for which $T(0,\neg\Bc^{\muZ}_{n_k})\le(1-\delta/2)n_k$.
For each such~$n_k$, there is a site $v_k$ such that $\muZ(v_k)>n_k$,
but $T(0,v_k)\le(1-\delta/2)n_k$. When $n_k$ is large we may further
assume that $\muZ(v_k)\le2n_k$. Consequently, we conclude that for
large $k$
%
\begin{equation}
\label{eqsubdiv} T(0,v_k)-\muZ(v_k) \le-\delta
n_k/2 \le-\delta\muZ(v_k)/4 \le-\varepsilon
\|v_k\|,
\end{equation}
for some $\varepsilon>0$. However, $T(0,\Delta_v)\le T(0,v)$, so the
occurrence of~(\ref{eqsubdiv}) for infinitely many $k$ is
contradicted by~(\ref{eqSTkesten}), almost surely. That is, $\Pr
(A_\delta)=0$ for every $\delta>0$, as required.
\end{pf}

\section{Large deviations below the time constant}\label{secLDbelow}

We will follow the approach of Kesten~\cite{kesten86}, Theorem~5.2, on
our way to a proof of Theorem~\ref{teoLDbelow}. If $\mu\equiv0$,
then there is nothing to prove. So, we may assume that $\mu\not\equiv
0$. Unlike Kesten, we will work with the point-to-shape travel times
introduced above in order to obtain a bound on deviations in all
directions simultaneously, and not only for coordinate directions. The
first and foremost step is this next lemma.

\begin{lma}\label{lmaTW1}
Let $X_{\ell,\ell+m}^{(q)}$ for $q=1,2,\ldots$ denote independent
random variables distributed as $T(\Bc^{\muZ}_\ell,\neg\Bc^{\muZ
}_{\ell+m})$. There exists $C<\infty$ such that for every $n\ge m\ge
\ell\ge1$ and $x>0$ we have
\[
\Pr\bigl(T\bigl(0,\neg\Bc^{\muZ}_n\bigr)<x \bigr) \le\sum
_{Q+1\ge
n/(m+C\ell)}n^{d-1} \biggl(C\frac{m}{\ell}
\biggr)^{ d(Q-1)} \Pr\Biggl(\sum_{q=1}^QX_{\ell,\ell+m}^{(q)}<x
\Biggr).
\]
\end{lma}

\begin{pf}
Pick $z\in\Z^d$ such that $\muZ(z)>n$. Let $\Gamma=\Gamma(z)$ be a
self-avoiding path connecting the origin to $z$. Choose a subsequence
$v_0,v_1,\ldots,v_Q$ of the vertices in $\Gamma$ as follows. Set
$v_0=0$. Given $v_q$, choose $v_{q+1}$ to be the first vertex in
$\Gamma$ succeeding $v_q$ such that
\[
\muZ(v_{q+1}-v_q)>m+2\ell.
\]
When no such vertex exists, stop and set $Q=q$. To find a lower bound
on $Q$, note that
\[
n < \muZ(z) \le\muZ(z-v_Q)+\muZ(v_Q) \le
\muZ(z-v_Q)+\sum_{q=0}^{Q-1}
\muZ(v_{q+1}-v_q).
\]
Since $\muZ(v_{q+1}-v_q)\le m+2\ell+\muZ(\ebf_1)$ and $\muZ
(z-v_Q)\le m+2\ell$, we see that $Q$ must satisfy
%
\begin{equation}
\label{eqQbound} n \le(Q+1) \bigl(m+ \bigl(2+\muZ(\ebf_1) \bigr)\ell
\bigr).
\end{equation}

Next, pick $r>0$ such that $[-r,r]^d\subseteq\Bc^{\muZ}_1$ and tile
$\Z^d$ with copies of $(-r\ell,r\ell]^d$ such that each box is
centered at a point in $\Z^d$, and each point in $\Z^d$ is contained
in precisely one box. Let $\Lambda_q$ denote the box that contains
$v_q$, and let $w_q$ denote the center of $\Lambda_q$. Of course, the
tiling can be assumed chosen such that $w_0=v_0=0$. Denote by $\Gamma
_q$ the part of the path $\Gamma$ that connects $v_q$ and $v_{q+1}$.
Note that for $q_1\neq q_2$ the two pieces $\Gamma_{q_1}$ and $\Gamma
_{q_2}$ are edge disjoint. By construction, $v_q$ is contained in the
copy of $\Bc^{\muZ}_\ell$ centered at $w_q$, while $v_{q+1}$ is not
contained in the copy of $\Bc^{\muZ}_{\ell+m}$ centered at $w_q$
(see Figure~\ref{figlower}). That is,
%
\begin{equation}
\label{eqpartitioncond} \muZ(v_q-w_q) \le\ell\quad\mbox{and}\quad
\muZ(v_{q+1}-w_q) > \ell+m.
\end{equation}
Moreover, the points $w_0,w_1,\ldots,w_{Q-1}$ have to satisfy
%
\begin{equation}
\label{eqpartitioncond2} \muZ(w_{q+1}-w_q) \le
m+4\ell+
\muZ(\ebf_1).
\end{equation}
Let $W_Q$ denote the set of all sequences $(w_0,w_1,\ldots,w_{Q-1})$
such that $w_0=0$, each $w_q$ is the center of some box $\Lambda_q$,
and $w_q$ and $w_{q+1}$ satisfies~(\ref{eqpartitioncond2}) for each
$q=0,1,\ldots,Q-2$.
%
\begin{figure}[t]

\includegraphics{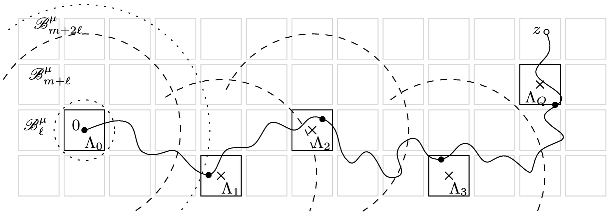}

\caption{The decomposition of a path into segments, where dots
represent the $v_k$'s.}
\label{figlower}
\end{figure}

Given $x>0$, $Q\in\Z_+$ and $w=(w_0,w_1,\ldots,w_{Q-1})\in W_Q$, let
$A(x,w)$ denote the event that there exists a path $\Gamma$ from the
origin to $z$ with edge disjoint segments $\Gamma_0,\Gamma_1,\ldots,\Gamma_{Q-1}$ such that:
\begin{longlist}[(2)]
\item[(1)] $\sum_{q=0}^{Q-1}T(\Gamma_q)<x$,
\item[(2)] the endpoints $v_q$ and $v_{q+1}$ of $\Gamma_q$ satisfy~(\ref
{eqpartitioncond}), for each $q=0,1,\ldots,Q-1$.
\end{longlist}
Since $T(\Gamma)\ge\sum_{q=0}^{Q-1}T(\Gamma_q)$, together
with~(\ref{eqQbound}), we obtain that
%
\begin{equation}
\label{eqpossiblepaths} \bigl\{T(0,z)<x \bigr\} \subseteq\bigcup
_{Q+1\ge n/(m+b\ell)} \bigcup_{w\in W_Q} A(x,w),
\end{equation}
where $b=2+\muZ(\ebf_1)$. Note that given $w_q$, the passage time of
any path between two vertices $v$ and $v'$ such that $\muZ(v-w_q)\le
\ell$ and $\muZ(v'-w_q)>\ell+m$ is stochastically larger than $T(\Bc
^{\muZ}_\ell,\neg\Bc^{\muZ}_{\ell+m})$. Hence, via a BK-like
inequality (e.g., Theorem~4.8, or~(4.13), in \cite{kesten86}), it is
for each $w\in W_Q$ possible to bound the probability of the event
$A(x,w)$ from above by
%
\begin{equation}
\label{eqdPX} \Pr\bigl(X_{\ell,\ell+m}^{(1)}+X_{\ell,\ell+m}^{(2)}+
\cdots+X_{\ell,\ell+m}^{(Q)}<x \bigr).
\end{equation}

It remains to count the number of elements $(w_0,w_1,\ldots,w_{Q-1})$
in $W_Q$. Assuming that $w_q$ has already been chosen, the number of
choices for $w_{q+1}$ is restricted by~(\ref{eqpartitioncond2}). In
particular, $w_{q+1}$ has to be contained in a cube centered at $w_q$
and whose side length is a multiple of $(5+\muZ(\ebf_1))m$. This cube
is intersected by at most $(Cm/\ell)^d$ boxes of the form $(-r\ell,r\ell
]^d$ in the tiling of $\Z^d$, for some $C<\infty$. Since
$w_{q+1}$ is the center of one of these boxes, this is also an upper
bound for its number of choices. Consequently,\vspace*{2pt} the total number of
choices for $w_1,w_2,\ldots,w_{Q-1}$ is at most $(Cm/\ell)^{d(Q-1)}$.
Together with~(\ref{eqpossiblepaths}) and~(\ref{eqdPX}), we
conclude that
\[
\Pr\bigl(T(0,z)<x \bigr) \le\sum_{Q+1\ge n/(m+C\ell)} \biggl(C
\frac{m}{\ell} \biggr)^{d(Q-1)}\Pr\Biggl(\sum
_{q=1}^QX_{\ell,\ell+m}^{(q)}<x \Biggr),
\]
for some $C<\infty$. The lemma follows observing that the number of
$z\in\Z^d$ that satisfies $\muZ(z)>n$ and has a neighbor within $\Bc
^{\muZ}_n$ is of order $n^{d-1}$.
\end{pf}

\begin{lma}\label{lmaTW2}
Assume that $\muZ\not\equiv0$. For every $\varepsilon>0$,
\[
\lim_{m\to\infty}\max_{\ell\le m}\Pr\bigl(T\bigl(
\Bc^{\muZ}_\ell,\neg\Bc^{\muZ}_{\ell+m}
\bigr)<m(1-\varepsilon) \bigr)=0.
\]
\end{lma}

\begin{pf}
Let $\Gamma$ be a path with endpoints $z$ and $y$ satisfying $\muZ
(z)\le\ell$ and $\muZ(y)>\ell+m$. Then
\[
T\bigl(0,\neg\Bc^{\muZ}_{\ell+m}\bigr) \le T(0,
\Delta_z)+\tbar|\Delta_z|+T(\Gamma),
\]
and we may choose $\Gamma$ so that $T(\Gamma)=T(\Bc^{\muZ}_\ell,\neg\Bc
^{\muZ}_{\ell+m})$. It follows that
\begin{eqnarray*}
T\bigl(0,\neg\Bc^{\muZ}_{\ell+m}
\bigr) &\le&\max\bigl\{T(0,\Delta_z)\dvtx \muZ(z)\le\ell\bigr\}
\\
&&{} +\tbar\cdot\max\bigl\{|\Delta_z|\dvtx \muZ(z)\le\ell\bigr\}+T\bigl(
\Bc^{\muZ
}_\ell,\neg\Bc^{\muZ}_{\ell+m}\bigr).
\end{eqnarray*}
For every $\varepsilon>0$, we have $\Pr(T(0,\neg\Bc^{\muZ
}_{\ell+m})<(\ell+m)(1-\varepsilon/4) )<\varepsilon$ for all
large enough $m$, by Proposition~\ref{propTshape}. Thus, it suffices
to show that for all large $m$ and $\ell\le m$ also
\[
\Pr\biggl(\max\bigl\{T(0,\Delta_z)\dvtx \muZ(z)\le\ell\bigr
\}>\ell+\frac
{\varepsilon m}{4} \biggr)+\Pr\biggl(\max\bigl\{|\Delta_z|\dvtx
\mu(z)\le m\bigr\} >\frac{\varepsilon m}{4\tbar} \biggr)
\]
is at most $\varepsilon$. The latter of the two probabilities
vanishes as $m\to\infty$ as a consequence of the exponential decay of
the diameter of a shell in~(\ref{eqshelldecay}). To show that the
former probability is small, we will use~(\ref{eqSTkesten}).

Let $N$ be an integer and note that if $\max\{T(0,\Delta_z)\dvtx \muZ
(z)\le\ell\}>\ell+\varepsilon m/4$, then either $\max\{T(0,\Delta
_z)\dvtx \muZ(z)\le N\}>\varepsilon m/4$ or $T(0,\Delta_z)-\muZ(z)>\ell
-\muZ(z)+\varepsilon m/4$ for some $z$ satisfying $\muZ(z)\in
[N,\ell]$. Since $m\ge\ell\ge\muZ(z)$, we have $\ell-\muZ
(z)+\varepsilon m/4\ge\varepsilon\muZ(z)/4$. We thus obtain the inequality
%
\begin{eqnarray*}
&&\Pr\biggl(\max\bigl\{T(0,\Delta_z)\dvtx
\muZ(z)\le\ell\bigr\}>\ell+\frac
{\varepsilon m}{4} \biggr)
\\
&&\qquad\le\Pr\biggl(\max\bigl\{T(0,\Delta_z)\dvtx \muZ(z)\le N\bigr\}>
\frac
{\varepsilon m}{4} \biggr)
\\
&&\quad\qquad{} +\Pr\bigl(T(0,\Delta_z)> (1+\varepsilon/4 )\muZ(z)\mbox{ for
some }\muZ(z)\ge N \bigr).
\end{eqnarray*}
From~(\ref{eqSTkesten}), we know that the right-hand side can be made
arbitrarily small by choosing $N$ large and sending $m$ to infinity.
\end{pf}

\begin{pf*}{Proof of Theorem~\ref{teoLDbelow}}
We may assume that $\muZ\not\equiv0$. We will prove that for every
$\varepsilon>0$ there exist $M=M(\varepsilon)$ and $\gamma=\gamma
(\varepsilon)$ such that for every $x\ge n\ge1$
\[
\Pr\bigl(T\bigl(0,\neg\Bc^{\muZ}_n\bigr)<n-\varepsilon x
\bigr) \le M e^{-\gamma x},
\]
from which Theorem~\ref{teoLDbelow} is an easy consequence.\vspace*{1pt}

Let\vspace*{1pt} $X_{\ell,\ell+m}^{(1)},X_{\ell,\ell+m}^{(2)},\ldots,X_{\ell,\ell
+m}^{(Q)}$ and $C<\infty$ be as in Lemma~\ref{lmaTW1}, and fix
$\varepsilon\in(0,4C)$. For some integer $m$, let $\ell=\ell(m)$
be the largest integer such that $\ell\le\frac{m\varepsilon}{4C}$.
Markov's inequality and independence give that for any $\xi>0$
\[
\Pr\Biggl(\sum_{q=1}^QX_{\ell,\ell+m}^{(q)}<n-
\varepsilon x \Biggr) \le e^{\xi(n-\varepsilon x)}\E\bigl[e^{-\xi
X_{\ell,\ell
+m}^{(1)}}
\bigr]^Q.
\]
Writing $n-\varepsilon x=n(1-\varepsilon)-\varepsilon(x-n)$, we
obtain for $(Q+1)(m+C\ell)\ge n$ the upper bound
\[
e^{-\varepsilon\xi(x-n)}e^{\xi(m+C\ell)} \bigl[e^{\xi(m+C\ell
)(1-\varepsilon)} \bigl(e^{-\xi m(1-\varepsilon/2)}+
\Pr\bigl(X_{\ell,\ell+m}^{(1)} <m(1-\varepsilon/2) \bigr) \bigr)
\bigr]^Q.
\]
Since $C\ell-m\varepsilon/2\le-m\varepsilon/4$ and $m+C\ell\le
(1+\varepsilon/4)m$, the expression within square brackets is at most
%
\begin{equation}
\label{eqsqbrackprob} e^{-\xi m\varepsilon/4}+e^{(1+\varepsilon
/4)\xi m} \Pr\bigl(X_{\ell,\ell+m}^{(1)}<m(1-
\varepsilon/2) \bigr).
\end{equation}
According to Lemma~\ref{lmaTW2}, we can make~(\ref{eqsqbrackprob})
arbitrarily small by choosing $\xi$ and $m$ such that $\xi m$ is large
and $m$ is as large as necessary. Fix $\xi$ and $m$ such that $\ell
\ge1$ and~(\ref{eqsqbrackprob}) is not larger than
\[
(2C)^{-d} \biggl(\frac{8C}{\varepsilon} \biggr)^{-d} \le
\biggl(2C\frac{m}{\ell} \biggr)^{-d}.
\]
Finally, apply Lemma~\ref{lmaTW1} with these $\xi$, $m$ and $\ell$
to obtain
\begin{eqnarray*}
&&\Pr\bigl(T\bigl(0,\neg\Bc^{\muZ}_n
\bigr)<n-\varepsilon x \bigr)
\\
&&\qquad\le e^{-\varepsilon\xi(x-n)}e^{\xi(m+C\ell)}\sum_{(Q+1)\ge
{n}/{(m+C\ell)}}n^{d-1}
\biggl(C\frac{m}{\ell} \biggr)^{d(Q-1)} \biggl(2C\frac{m}{\ell}
\biggr)^{-dQ}
\\
&&\qquad\le e^{-\varepsilon\xi(x-n)}\cdot e^{\xi(m+C\ell)}\cdot n^{d-1}
\cdot2^{-d(n/(m+C\ell)-1)+1},
\end{eqnarray*}
which is of the required form.
\end{pf*}

\section{A regenerative approach}\label{secreg}

We will in this section explore a regenerative approach that can be
used to study the asymptotics of travel times along cylinders. This
approach was previously studied in more detail in \cite{A-1}. It will
for the sake of this paper be sufficient to obtain a sequence which is
approximately regenerative, which in turn avoids some additional
technicalities. Some additional notation will be required however.

Given\vspace*{2pt} $z\in\Z^d$ and $r\ge0$, let $\Cyl(z,r):=\bigcup_{a\in\R
}B(az,r)$ denote the cylinder in direction $z$ of radius $r$, where
$B(x,r):=\{y\in\R^d\dvtx |y-x|\le r\}$ denotes the closed Euclidean ball.
The travel time between two points $x$ and $y$ over paths restricted to
the cylinder $\Cyl(z,r)$ will be denoted by $T_{\Cyl(z,r)}(x,y)$. The
regenerative approach referred to will consist of a comparison between
$T_{\Cyl(z,r)}(0,nz)$ and the sum of travel times between randomly
chosen ``cross-sections'' of $\Cyl(z,r)$.

%
\begin{figure}[b]

\includegraphics{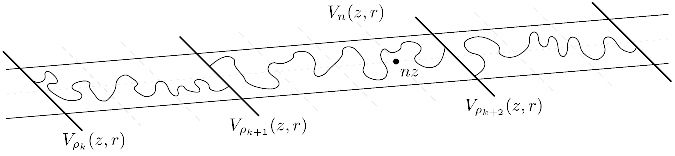}

\caption{A piece of $\Cyl(z,r)$. The thick diagonal lines indicate $\{
V_{\rho_j}(z,r)\}_{j\ge0}$ and the curly lines $\{T_{\Cyl
(z,r)}(V_{\rho_{j-1}},V_{\rho_j})\}_{j\ge1}$. In this illustration,
we have $\nu(n)=k+2$.}
\label{figreg}
\end{figure}

Due to symmetry it means no restriction assuming that $z\in\Z^d$ lies
in the first orthant, that is, that the coordinate $z_i\ge0$ for each
$i=1,2,\ldots,d$. Let $\Hb_n:=\{z\in\Z^d\dvtx  z_1+z_2+\cdots+z_d=n\}$,
$r\ge0$, and pick $\tbar\in\R_+$ such that $\Pr(\tau_e\le\tbar
)>0$. The following notation will be used, and is illustrated in Figure \ref{figreg} below:
\begin{eqnarray*}
 V_n(z,r)&:=&\Cyl(z,r)\cap
\Hb_{n\|z\|},
\\
E_n(z,r)&:=&\bigl\{\mbox{edges connecting }\Cyl(z,r)\cap
\Hb_{n\|z\|
-1}\mbox{ with }\Cyl(z,r)\cap\Hb_{n\|z\|}\bigr\},
\\
A_n(z,r)&:=&\bigl\{\tau_e\le\tbar\mbox{ for all }e\in
E_n(z,r)\bigr\},
\\
\rho_j(z,r)&:=&\min\bigl\{n>\rho_{j-1}(z,r)\dvtx A_n(z,r)
\mbox{ occurs}\bigr\} \qquad\mbox{for }j\ge1, \rho_0=0.
\end{eqnarray*}
When understood from the context, the reference to $z$ and $r$ will be dropped.

Note that $\{A_n(z,r)\}_{n\geq1}$ are i.i.d., so the increments $\{
\rho_j-\rho_{j-1}\}_{j\ge1}$ are independent geometrically
distributed with success probability $\Pr(A_0(z,r))$. Consequently, $\{
T_{\Cyl(z,r)}(V_{\rho_{j-1}},V_{\rho_j})\}_{j\ge1}$ are i.i.d.
Introduce the following notation for their means:
\begin{eqnarray*}
\mu_\tau(z,r)&:=&\E\bigl[T_{\Cyl(z,r)}(V_{\rho_0},V_{\rho_1})
\bigr],
\\
\mu_\rho(z,r)&:=&\E[\rho_1-\rho_0],
\end{eqnarray*}
and, for the time constant for travel times restricted to cylinders, let
\[
\mu_{\Cyl(z,r)}:=\lim_{n\to\infty}\frac{\E[T_{\Cyl(z,r)}(V_0,V_n)]}{n}.
\]
The existence of the above limit is given by Fekete's lemma, since that 
$ (-\E[T_{\Cyl(z,r)}(V_0,V_n)] )_{n\ge1}$ is a subadditive
sequence. A sufficient condition for the limit $\mu_{\Cyl(z,r)}$ to
be finite will be achieved with Proposition~\ref{propregincrement} below.

Finally, travel times on $\Cyl(z,r)$ and the sequence $\{T_{\Cyl
(z,r)}(V_{\rho_{j-1}},V_{\rho_j})\}_{j\ge1}$ will be compared via
optimal stopping. We define for that purpose the stopping time
\[
\nu(m)=\nu(m,z,r):=\min\bigl\{j\ge1\dvtx  \rho_j(z,r)>m\bigr\}.
\]
Note that $\nu(m)-1$ equals the number of $n\in\{1,2,\ldots,m\}$ for
which $A_n(z,r)$ occurs, which is binomially distributed with success
probability $\Pr(A_0(z,r))=\mu_\rho(z,r)^{-1}$.

\begin{remark*}
A geometrical constraint should be noted. For some $z\in\Z^d$, there
may not be any paths at all between $x$ and $y$ that only passes
through points in $\Cyl(z,r)$, when $r$ is small. [In this case set
$T_{\Cyl(z,r)}(x,y)=\infty$.] However, it is not hard to realize that
for every $k\ge1$, there is $R=R(d,k)$ such that for every $z\in\Z
^d$ and $r\ge R$ there are $k$ edge-disjoint paths from $V_0(z,r)$ to
$V_1(z,r)$ of length $\|z\|$, which are all contained in $\Cyl(z,r)$.
($R=k\sqrt{d}$ is sufficient.)
\end{remark*}

\subsection{Tail and moment comparisons}

The next task will be to relate tail probabilities of travel times and
moments of $T_{\Cyl(z,r)}(V_{\rho_0},V_{\rho_1})$ with the
corresponding quantities for $Y$. The latter will provide a sufficient
condition for $\mu_{\Cyl(z,r)}$ to be finite and converge to $\muZ
(z)$ as $r\to\infty$. We begin with a well-known tail comparison.

\begin{lma}\label{lmaYcond}
For every $z\in\Z^d$, $x\ge0$ and large enough $r$,
\[
\Pr\bigl(T_{\Cyl(z,r)}(0,z)>9\|z\|x \bigr) \le9^{2d}\|z\|
\Pr(Y>x).
\]
\end{lma}

\begin{pf}
Note that there are $2d$ edge disjoint paths between the origin and
$\ebf_1$ of length at most 9. Denote these paths by $\Gamma_1,\Gamma
_2,\ldots,\Gamma_{2d}$, and assume that $\Gamma_1$ is the longest
among them. Clearly,
\[
\Pr\Bigl(\min_{i=1,2,\ldots,2d}T(\Gamma_i)>9x \Bigr) \le
\Pr\bigl(T(\Gamma_1)>9x \bigr)^{2d} \le9^{2d}
\Pr(\tau_e>x)^{2d}.
\]
For $r$ large, $T_{\Cyl(z,r)}(0,z)$ is dominated by $\|z\|$ random
variables distributed as $\min_{i=1,2,\ldots,2d}T(\Gamma_i)$. Consequently,
\begin{eqnarray*}
\Pr\bigl(T_{\Cyl(z,r)}(0,z)>9\|z\|x \bigr) &\le&
\|z\| \Pr\Bigl(\min_{i=1,2,\ldots,2d}T(\Gamma_i)>9x \Bigr)
\\
&\le&9^{2d}\|z\| \Pr(Y>x),
\end{eqnarray*}
as required.
\end{pf}

In preparation for the second aim, we have a couple of lemmata of
general character.

\begin{lma}\label{lmaYcomp1}
Let $\{\tau_i\}_{i\ge1}$ be a collection of nonnegative i.i.d. random
variables. For any $\alpha,\beta>0$ and integers $L\ge K\ge1$
such that $\beta K\le\alpha L$, then
\[
\E\Bigl[ \Bigl(\min_{i\le L}\tau_i
\Bigr)^\beta\Bigr] \le1+\frac
{\beta}{\alpha}\E\Bigl[ \Bigl(\min
_{i\le K}\tau_i \Bigr)^\alpha
\Bigr]^{L/K}.
\]
\end{lma}

\begin{pf}
Recall the formula $\E[X^\alpha]=\alpha\int_0^\infty x^{\alpha
-1}\Pr(X>x) \,dx$, valid for nonnegative random variables and $\alpha
>0$. Note that for any $x\ge1$ Markov's inequality gives that
\[
\Pr(\tau_i>x) = \Pr\Bigl(\min_{i\le K}
\tau_i>x \Bigr)^{1/K} \le\frac{\E[ (\min_{i\le K}\tau_i )^\alpha
]^{1/K}}{x^{\alpha/K}},
\]
from which one, under the imposed conditions, easily obtains
\[
x^{\beta-1}\Pr(\tau_i>x)^L \le x^{\alpha-1}\Pr
\Bigl(\min_{i\le K}\tau_i>x \Bigr)\cdot\E\Bigl[
\Bigl(\min_{i\le K}\tau_i \Bigr)^\alpha
\Bigr]^{(L-K)/K}.
\]
Finally, integrating over the intervals $[0,1)$ and $[1,\infty)$
separately yields
\begin{eqnarray*}
\E\Bigl[ \Bigl(\min_{i\le L}
\tau_i \Bigr)^\beta\Bigr] &\le&1+\beta\E\Bigl[ \Bigl(\min
_{i\le K}\tau_i \Bigr)^\alpha
\Bigr]^{(L-K)/K}\int_{x\ge1}x^{\alpha-1}\Pr\Bigl(\min
_{i\le K}\tau_i>x \Bigr) \,dx
\\
&=& 1+\frac{\beta}{\alpha}\E\Bigl[ \Bigl(\min_{i\le K}
\tau_i \Bigr)^\alpha\Bigr]^{1+(L-K)/K},
\end{eqnarray*}
as required.
\end{pf}

\begin{lma}\label{lmaYcomp2}
Let $\{\tau_{i,j}\}_{i,j\ge1}$ be a collection of nonnegative i.i.d.
random variables. For any $\alpha,\beta>0$ and integers $K,N\ge1$
and $L\ge K$ satisfying $\beta K\le\alpha L$,
\[
\E\biggl[ \biggl(\min_{i\le L}\sum
_{j\le N}\tau_{i,j} \biggr)^\beta\biggr] \le
N^{L+\beta} \biggl(1+\frac{\beta}{\alpha}\E\Bigl[ \Bigl(\min
_{i\le K}\tau_{i,j} \Bigr)^\alpha
\Bigr]^{L/K} \biggr).
\]
\end{lma}

\begin{pf}
First, since if a sum of $N$ nonnegative numbers is greater than $x$,
then at least one of the terms has to be greater than $x/N$, it follows that
\[
\Pr\biggl(\min_{i\le L}\sum_{j\le N}
\tau_{i,j}>x \biggr) = \Pr\biggl(\sum_{j\le N}
\tau_{i,j}>x \biggr)^L \le N^L \Pr(\tau
_{i,j}>x/N )^L.
\]
Thus, via the substitution $x=Ny$, we conclude that
\begin{eqnarray*}
\E\biggl[ \biggl(\min_{i\le L}\sum
_{j\le N}\tau_{i,j} \biggr)^\beta
\biggr] &\le&\beta N^L\int_0^\infty
x^{\beta-1}\Pr(\tau_{i,j}>x/N )^L \,dx
\\
&=& N^{L+\beta}\E\Bigl[ \Bigl(\min_{i\le L}
\tau_{i,j} \Bigr)^\beta\Bigr],
\end{eqnarray*}
from which the statement follows via Lemma~\ref{lmaYcomp1}.
\end{pf}

\begin{prop}\label{propregincrement}
For every $\beta\ge\alpha>0$ and $z\in\Z^d$, there is a finite
constant $R_1=R_1(\alpha,\beta,d)$ such that for $r\ge R_1$ and some
finite constant $M_1=M_1(\alpha,\beta,d,z,r)$,
\[
\E\bigl[T_{\Cyl(z,r)}(V_{\rho_0},V_{\rho_1})^\beta
\bigr]\le M_1 \bigl(1+\E\bigl[Y^\alpha\bigr]
\bigr)^{\beta/\alpha+1}.
\]
\end{prop}

\begin{pf}
If $\Pr(\tau_e>\tbar)=0$, then $\rho_1=1$ and the statement is a
consequence of Lemma~\ref{lmaYcomp2}. Assume instead the contrary, in
which case a bit more care is needed before appealing to Lemma~\ref
{lmaYcomp2}.

Let $\eta=\{\eta_e\}_{e\in\Ec}$ denote the family of indicator
functions $\eta_e=\ind_{\{\tau_e>\tbar\}}$. Independently of $\{
\tau_e\}_{e\in\Ec}$, let $\{\tilde{\tau}_e\}_{e\in\Ec}$ be a
collection of independent random variables distributed as $\Pr(\tilde
{\tau}_e\in\cdot)=\Pr(\tau_e\in\cdot|\tau_e>\tbar)$, and
define $\{\sigma_e\}_{e\in\Ec}$ as
\[
\sigma_e:=\cases{ \tau_e, &\quad if $
\eta_e=1$,
\cr
\tilde{\tau}_e, &\quad if $
\eta_e=0$.}
\]
Note that $\{\sigma_e\}_{e\in\Ec}$ is an i.i.d. family independent
of $\eta$, but that $\eta$ determines $\{A_n(z,r)\}_{n\geq1}$, and
hence $\{\rho_j-\rho_{j-1}\}_{j\ge1}$, for every $z$ and $r$. In
particular, $\{\sigma_e\}_{e\in\Ec}$ and $\{\rho_j-\rho_{j-1}\}
_{j\ge1}$ are independent. Let $T_{\Cyl}'(x,y)$ denote the passage
time between $x$ and $y$ with respect to $\{\sigma_e\}_{e\in\Ec}$.
By construction $\tau_e\le\sigma_e$ for every $e\in\Ec$, so
$T_{\Cyl}(x,y)\le T_{\Cyl}'(x,y)$.

Fix $\beta\ge\alpha>0$ and $z\in\Z^d$. Choose $r=r(\alpha,\beta,d)$
large so that there are at least $2d\beta/\alpha$ disjoint paths
between $V_0(z,r)$ and $V_1(z,r)$ of length $\|z\|$, contained in $\Cyl
(z,r)$. Similarly, there are equally many paths between $V_{\rho
_0}(z,r)$ and $V_{\rho_1}(z,r)$ of length $(\rho_1-\rho_0)\|z\|$.
Hence, by Lemma~\ref{lmaYcomp2}
\begin{eqnarray*}
&& \E\bigl[T_\Cyl'(V_{\rho_0},V_{\rho_1})^\beta
|\eta\bigr]
\\
&&\qquad \le\bigl((\rho_1-\rho_0)\|z\|
\bigr)^{2d\beta/\alpha+\beta+1} \biggl(1+\frac{\beta}{\alpha}\E\Bigl[
\Bigl(\min
_{i\le2d}\sigma_i \Bigr)^\alpha
\Bigr]^{\beta/\alpha+1/(2d)} \biggr),
\end{eqnarray*}
where $\sigma_1,\sigma_2,\ldots,\sigma_{2d}$ denote independent
variables distributed as $\sigma_e$. In addition,
\begin{eqnarray*}
\E\Bigl[ \Bigl(\min_{i\le2d}
\sigma_i \Bigr)^\alpha\Bigr] &=& \alpha\int
_0^\infty x^{\alpha-1} \Pr(
\tau_e>x|\tau_e>\tbar)^{2d}\,dx
\\
&\le&\E\Bigl[ \Bigl(\min_{i\le2d}\tau_i
\Bigr)^\alpha\Bigr]\Pr(\tau_e>\tbar)^{-2d}.
\end{eqnarray*}
Since $T_\Cyl(V_{\rho_0},V_{\rho_1})\le T_\Cyl'(V_{\rho_0},V_{\rho
_1})$, and $\rho_1-\rho_0$ is geometrically distributed, the bound
follows easily.\vadjust{\goodbreak}
\end{pf}

\subsection{Time constant comparison}

Proposition~\ref{propregincrement} gives, in particular, a criterion
for $\mu_\tau(z,r)$ and $\mu_{\Cyl(z,r)}$ to be finite.

\begin{lma}\label{lmaregmu}
Assume that $\E[Y^\alpha]<\infty$ for some $\alpha>0$. Then $\mu
_\tau(z,r)$ is finite for all $z\in\Z^d$ and $r\ge R_1$, where $R_1$
is given by Proposition~\ref{propregincrement} (with $\beta=1$). Moreover,
\[
\frac{\mu_\tau(z,r)}{\mu_\rho(z,r)} \le\mu_{\Cyl(z,r)} \le\frac{\mu
_\tau(z,r)}{\mu_\rho(z,r)}+
\frac{\tbar|E_n|}{\mu
_\rho(z,r)}.
\]
\end{lma}

\begin{pf}
The former assertion is the content of Proposition~\ref
{propregincrement}. For the latter, note that
\begin{eqnarray*}
\sum_{j=1}^{\nu(n)-1}T_{\Cyl(z,r)}(V_{\rho
_{j-1}},V_{\rho_j})
&\le& T_{\Cyl(z,r)}(V_0,V_n)
\\
&\le&\sum_{j=1}^{\nu(n)}T_{\Cyl(z,r)}(V_{\rho_{j-1}},V_{\rho
_j})+
\tbar|E_n|\bigl(\nu(n)-1\bigr).
\end{eqnarray*}
Taking expectations, the right-hand side is via Wald's lemma turned into
\[
\E\bigl[\nu(n)\bigr]\E\bigl[T_{\Cyl(z,r)}(V_{\rho_0},V_{\rho_1})
\bigr]+\tbar|E_n|\E\bigl[\nu(n)-1\bigr],
\]
which after division by $n$ gives
\[
\biggl(\frac{1}{\mu_\rho(z,r)}+\frac{1}{n} \biggr)\mu_\tau(z,r)+
\frac{\tbar|E_n|}{\mu_\rho(z,r)}.
\]
Sending $n$ to infinity thus gives the upper bound.

For the lower bound, it suffices to show that $\E[T_{\Cyl
(z,r)}(V_{\rho_{\nu(n)-1}},V_{\rho_{\nu(n)}}) ]$ is bounded.
This follows from $T_{\Cyl(z,r)}(V_{\rho_{\nu(n)-1}},V_{\rho_{\nu
(n)}})$ being stochastically dominated by $T_{\Cyl(z,r)}(V_{\rho
_{-1}},V_{\rho_1})$, where $\rho_{-1}=\max\{n\le0\dvtx A_n(z,r)
\mbox{ occurs}\}$. Now, we have not proven that $T_{\Cyl
(z,r)}(V_{\rho_{-1}},V_{\rho_1})$ has finite mean, but that follows
readily from the proof of Proposition~\ref{propregincrement}.
\end{pf}

Since $\mu_{\Cyl(z,r)}$ is decreasing in $r$, it seems reasonable
that it should converge to its lower bound $\muZ(z)$, as $r$ tends to
infinity. An indication of this was given already in~\cite{chacha84}
and~\cite{kesten86}, but proofs of this fact may have appeared only
more recently, see, for example, \cite{A08thesis,A-1}. These proofs assume
finite expectation of $Y$, and to extend them to a minimal moment
condition requires some care.

\begin{prop}\label{proptubeconstant}
Assume that $\E[Y^\alpha]<\infty$ for some $\alpha>0$. Then, for
every $z\in\Z^d$,
\[
\lim_{r\to\infty}\mu_{\Cyl(z,r)}=\muZ(z).
\]
\end{prop}

\begin{pf}
Fix $z\in\Z^d$ and, based on Proposition~\ref{propregincrement},
pick $s>0$ sufficiently large for $\E[T_{\Cyl(z,s)}(V_{\rho
_0(z,s)},V_{\rho_1(z,s)})^2]$ to be finite. To the end of this proof,
let $V_n=V_n(z,s)$, $E_n=E_n(z,s)$, $\rho_j=\rho_j(z,s)$, $\nu
(n)=\nu(n,z,s)$, and $\epsilon(s)=\tlarge|E_0(z,s)|$. For $r\ge
s$, let
\[
a_n(r,s):=\E\bigl[T_{\Cyl(z,r)}(V_{\rho_1(z,s)},V_{\rho_{\nu
(n)}(z,s)})
\bigr]+\epsilon(s).
\]
The sequence $ (a_n(r,s) )_{n\ge1}$ is subadditive, that is,
$a_{n+m}(r,s)\le a_n(r,s)+a_m(r,s)$ for all $n,m\ge1$, as a
consequence of the inequality (note that $\rho_1=\rho_{\nu(0)}$)
\begin{eqnarray*}
\E\bigl[T_{\Cyl(z,r)}(V_{\rho_1},V_{\rho_{\nu(n+m)}})
\bigr]+\epsilon(s) &\le&\E\bigl[T_{\Cyl(z,r)}(V_{\rho_{\nu
(0)}},V_{\rho_{\nu(n)}})
\bigr]
\\
&&{} +\E\bigl[T_{\Cyl(z,r)}(V_{\rho_{\nu(n)}},V_{\rho_{\nu
(n+m)}}) \bigr]+2
\epsilon(s).
\end{eqnarray*}
Recall Fekete's lemma, which says that for any subadditive sequence the
limit $\lim_{n\to\infty}\frac{1}{n}a_n(r,s)$ exists and equals
$\inf_{n\ge1}\frac{1}{n}a_n(r,s)$. This holds for every $r\ge s$,
including the case $r=\infty$ in which the cylinder equals the whole lattice.

Next note the inequality
\begin{eqnarray*}
T_{\Cyl(z,r)}\bigl(V_0(z,r),V_n(z,r)
\bigr) &\le& T_{\Cyl
(z,r)}\bigl(V_0(z,s),V_n(z,s)
\bigr)
\\
&\le& T_{\Cyl(z,r)}(V_0,V_{\rho_1})+T_{\Cyl(z,r)}(V_{\rho
_1},V_{\rho_{\nu(n)}})
\\
&&{} +T_{\Cyl(z,r)}(V_n,V_{\rho_{\nu(n)}})+2\epsilon(s),
\end{eqnarray*}
which shows that $\mu_{\Cyl(z,r)}\le\lim_{n\to\infty}\frac
{1}{n}a_n(r,s)$ for every $r\ge s$. Consequently,
\begin{eqnarray*}
\lim_{r\to\infty}\mu_{\Cyl(z,r)} &=&
\inf_{r\ge0}\mu_{\Cyl
(z,r)} \le\inf_{r\ge s}
\inf_{n\ge1}\frac{a_n(r,s)}{n} = \inf_{n\ge1}\inf
_{r\ge s}\frac{a_n(r,s)}{n}
\\
&=& \inf_{n\ge1}\lim_{r\to\infty}\frac{a_n(r,s)}{n} =
\lim_{n\to\infty}\frac{\E[T(V_{\rho_1},V_{\rho_{\nu(n)}})]}{n},
\end{eqnarray*}
where we in the second to last step have used that $T_{\Cyl
(z,r)}(\cdot,\cdot)$ is decreasing in $r$, and in the finial step
have appealed to the monotone convergence theorem and used that
$(a_n(\infty,s))_{n\ge1}$ is subadditive. It remains to prove that
the limit equals $\muZ(z)$.

We will proceed by showing that $\frac{1}{n}T(V_{\rho_1},V_{\rho
_{\nu(n)}})\to\muZ(z)$ in probability as $n\to\infty$, and then
argue that the limit carries over in the mean due to uniform
integrability of the family $\{\frac{1}{n}T(V_{\rho_1},V_{\rho_{\nu
(n)}})\}_{n\ge1}$. We begin proving convergence in probability. By
subadditivity,
\[
\bigl|T(V_{\rho_1},V_{\rho_{\nu(n)}})-T(0,nz) \bigr| \le T(0,V_{\rho_{\nu
(0)}})+T(nz,V_{\rho_{\nu(n)}}).
\]
Since the distributions of the dominating two terms are independent of
$n$, the convergence of $\frac{1}{n}T(V_{\rho_1},V_{\rho_{\nu
(n)}})$ in probability to $\muZ(z)$ follows from the convergence of
$\frac{1}{n}T(0,nz)$ in~(\ref{deftimeconstant}).

The\vspace*{1pt} condition $\sup_{n\ge1}\E[ (\frac{1}{n}T(V_{\rho
_1},V_{\rho_{\nu(n)}}) )^\alpha]<\infty$, for some $\alpha
>1$, is sufficient for uniform integrability. To see that this holds,
observe that
\[
T(V_{\rho_1},V_{\rho_{\nu(n)}}) \le T_{\Cyl(z,s)}(V_{\rho
_1},V_{\rho_{n+1}})
\le\sum_{j=2}^{n+1}T_{\Cyl(z,s)}(V_{\rho
_{j-1}},V_{\rho_j})+n
\epsilon(s),
\]
where we have used that $\nu(n)-1\le n$. Thus, by convexity of the
function $x^2$, we find that
\[
\biggl(\frac{1}{n}T(V_{\rho_1},V_{\rho_{\nu(n)}})
\biggr)^2 \le2 \Biggl(\frac{1}{n}\sum
_{j=2}^{n+1}T_{\Cyl(z,s)}(V_{\rho
_{j-1}},V_{\rho_j})^2+
\epsilon(s)^2 \Biggr).
\]
Since the terms in the sum are i.i.d., we obtain
\[
\E\biggl[ \biggl(\frac{1}{n}T(V_{\rho_1},V_{\rho_{\nu(n)}})
\biggr)^2 \biggr] \le2\E\bigl[T_{\Cyl(z,s)}(V_{\rho_0},V_{\rho
_1})^2
\bigr]+2\epsilon(s)^2,
\]
which is finite and independent of $n$. Thus, $\{\frac{1}{n}T(V_{\rho
_1},V_{\rho_{\nu(n)}})\}_{n\ge1}$ is uniformly integrable, and
\[
\lim_{n\to\infty}\frac{\E[T(V_{\rho_1},V_{\rho_{\nu
(n)}})]}{n}=\muZ(z),
\]
as required.
\end{pf}

\section{Large deviations above the time constant}\label{secLDabove}

In this section, we estimate the probability of large deviations above
the time constant and prove Theorem~\ref{teoLDabove}. Recall that it
suffices to consider $z$ in the first orthant, due to symmetry. The
regenerative approach set up for in the previous section will serve to
obtain a first modest estimate on the tail decay. This first step of
the proof is as follows.

\begin{lma}\label{lmaLDtail}
Assume that $\E[Y^\alpha]<\infty$ for some $\alpha>0$. There exists
$R_2=R_2(\alpha,d)$ such that for every $\varepsilon>0$, $z\in\N
^d$ and $r\ge R_2$, there is a finite constant $M_2=M_2(\alpha
,\varepsilon,d,z,r)$ such that for every $n\in\N$ and $x\ge n$
\[
\Pr\bigl(T_{\Cyl(z,r)}(\Hb_0,\Hb_{n\|z\|})-n
\mu_{\Cyl
(z,r)}>\varepsilon x\|z\| \bigr) \le\frac{M_2}{x}.
\]
\end{lma}

\begin{pf}
We may without loss of generality assume that $\alpha\in(0,2]$. Let
$\beta=2$ and let $R_1=R_1(\alpha,d)$ be given as in Proposition~\ref
{propregincrement}. In particular, $\mu_{\Cyl(z,r)}$ is finite for
$r\ge R_1$. Fix $r\ge R_1$, $\varepsilon>0$ and choose $N\in\N$
large enough for $2\tbar|E_0(z,r)|\le\varepsilon N\|z\|$ to hold.
Set $y=Nz$ and let $m_n=\max\{m\ge0\dvtx mN\le n\}$. To the end of this\vadjust{\goodbreak}
proof, let $V_n=V_n(y,r)$, $E_n=E_n(y,r)$, $\rho_j=\rho_j(y,r)$, $\nu
(n)=\nu(n,y,r)$, $\mu_\tau=\mu_\tau(y,r)$, and $\mu_\rho=\mu
_\rho(y,r)$. (For a relation between $nz$, $m_ny$ and $\rho_{\nu
(m_n)}y$, see Figure~\ref{figsparse}.) Recall that $\rho_0=0$, so
$V_{\rho_0}(y,r)=V_0(z,r)$. By subadditivity,
\begin{eqnarray*}
T_{\Cyl(z,r)}(\Hb_0,
\Hb_{n\|z\|})-n\mu_{\Cyl(z,r)} &\le&\sum_{j=1}^{\nu(m_n)}
\bigl(T_{\Cyl(z,r)}(V_{\rho_{j-1}},V_{\rho_j})-\mu_\tau
\bigr)
\\
&&{} +T_{\Cyl(z,r)}(V_{\rho_{\nu(m_n)}},\Hb_{n\|z\|})
\\
&&{} + \bigl(\nu(m_n)\mu_\tau-n\mu_{\Cyl(z,r)} \bigr)+
\sum_{j=1}^{\nu
(m_n)}\tbar|E_0|.
\end{eqnarray*}
Label the four terms on the right-hand side as $X_1,X_2,X_3,X_4$. Since
that $\sum_{i\le4}X_i>4\varepsilon x\|z\|$ would imply that
$X_i>\varepsilon x\|z\|$ for some $i=1,2,3,4$, and since $\varepsilon
>0$ was arbitrary, it suffices to obtain a bound on $\Pr
(X_i>\varepsilon x\|z\|)$ of the desired form, for each $i=1,2,3,4$ separately.
%
\begin{figure}

\includegraphics{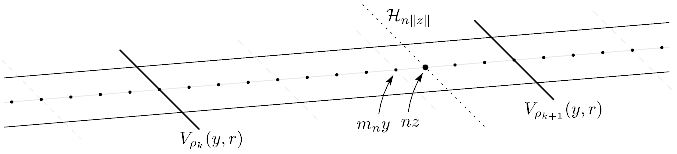}

\caption{Dots indicate $\{jz\}_{j\ge0}$, dashed diagonal lines $\{
V_j(y,r)\}_{j\ge0}$, and thick diagonal lines $\{V_{\rho_j}(y,z)\}
_{j\ge0}$. In the illustration we have $y=4z$ and $\nu(m_n)=\nu
(m_n,y,r)=k+1$.}
\label{figsparse}
\end{figure}

Starting from behind, since $\nu(m_n)\le m_n+1\le n/N+1$, it follows
that for $n\ge N$, $\Pr(\nu(m_n)\tbar|E_0|>\varepsilon n\|z\|
)=0$ by the choice of $N$. So, the last term satisfies a bound on
the desired form. Since $\mu_{\Cyl(y,r)}=N\mu_{\Cyl(z,r)}$, the
third term is via Lemma~\ref{lmaregmu} bounded above by
\[
\nu(m_n)\mu_\rho\mu_{\Cyl(y,r)}-m_n
\mu_{\Cyl(y,r)}.
\]
Recall that $\nu(m)-1$ counts the number of $k\in\{1,2,\ldots,m\}$
for which $A_k(y,r)$ occurs, and is therefore binomially distributed
with success probability\break $\Pr(A_n(y,r) )=1/\mu_\rho$. For
large $n$, we will have $\varepsilon x\|z\|/2>\mu_\rho\mu_{\Cyl
(y,r)}$. Thus, for large $n$ Chebyshev's inequality may be applied to give
\[
\Pr\biggl(\nu(m_n)\mu_\rho-m_n>\varepsilon
\frac{x\|z\|}{\mu
_{\Cyl(y,r)}} \biggr) \le
4\mu_{\Cyl(y,r)}^2
\frac{\mu_\rho-1}{\varepsilon^2N\|z\|^2x},
\]
which also meets the requirement.

For $\beta=2$ and $y=Nz$, let $M_1=M_1(\alpha,d,y,r)$ be given as in
Proposition~\ref{propregincrement}. Since $m_n\|y\|\le n\|z\|<\rho
_{\nu(m_n)}\|y\|$, then $T_{\Cyl(z,r)}(\Hb_{n\|z\|},V_{\rho_{\nu
(m_n)}})\le T_{\Cyl(z,r)}(V_{m_n},V_{\rho_{\nu(m_n)}})$, which is
distributed as $T_{\Cyl(z,r)}(V_{\rho_0},V_{\rho_1})$. Consequently,
Markov's inequality and Proposition~\ref{propregincrement} give that
\[
\Pr\bigl(T_{\Cyl(z,r)}(V_{\rho_{\nu(m_n)}},\Hb_{n\|z\|
})>\varepsilon x
\|z\| \bigr) \le M_1\frac{(1+\E[Y^\alpha
])^{1+1/\alpha}}{\varepsilon\|z\|x},
\]
for $r\ge R_1$. Finally, by Wald's lemma $\sum_{j=1}^{\nu(m_n)}
(T_{\Cyl(z,r)}(V_{\rho_{j-1}},V_{\rho_j})-\mu_\tau)$ has mean
zero and second moment
\[
\E\Biggl[ \Biggl(\sum_{j=1}^{\nu(m_n)}
\bigl(T_{\Cyl(z,r)}(V_{\rho
_{j-1}},V_{\rho_j})-\mu_\tau
\bigr) \Biggr)^2 \Biggr] = \Var\bigl(T_{\Cyl(z,r)}(V_{\rho_0},V_{\rho_1})
\bigr)\E\bigl[\nu(m_n) \bigr].
\]
Using Chebyshev's inequality, Proposition~\ref{propregincrement} and
the identity $\E[\nu(m_n)]=1+m_n\mu_\rho^{-1}$ shows that $\Pr
(\sum_{j=1}^{\nu(m_n)} (T_{\Cyl(z,r)}(V_{\rho_{j-1}},V_{\rho
_j})-\mu_\tau)>\varepsilon x\|z\| )$ is bounded above by
\[
M_1\frac{(1+\E[Y^\alpha])^{1+2/\alpha}(\mu_\rho
^{-1}N^{-1}+x^{-1})}{\varepsilon^2\|z\|^2x},
\]
for all $r\ge R_1$, as required.
\end{pf}

In the second step, we improve upon the above decay by aligning
disjoint cylinders.

\begin{prop}\label{propLDtail}
Assume that $\E[Y^\alpha]<\infty$ for some $\alpha>0$. For every
$\varepsilon>0$, $q\ge1$ and $z\in\Z^d$ there exists
$M_3=M_3(\varepsilon,\alpha,q,d,z)$ such that for all $n\in\N$ and
$x\ge n$
\[
\Pr\bigl(T(0,nz)-n\muZ(z)>\varepsilon x\|z\| \bigr) \le M_3 \Pr
(Y> x/M_3)+\frac{M_3}{x^q}.
\]
\end{prop}

\begin{pf}
We may assume that $z$ lies in the first orthant due to symmetry. Fix
$\varepsilon>0$, $q\in\Z_+$ and choose $r\ge R_2$ large enough for
$\mu_{\Cyl(z,r)}-\muZ(z)\le\varepsilon\|z\|$ to hold, where
$R_2=R_2(\alpha,d)$ is as in Lemma~\ref{lmaLDtail}. Pick
$v_1,v_2,\ldots,v_q\in\Hb_0$ such that the transposed cylinders
$v_i+\Cyl(z,r)$ are pairwise disjoint, and choose $s>r$ so that
$v_i+\Cyl(z,r)\subseteq\Cyl(z,s)$ for all $i=1,2,\ldots,q$. For
such $s$, the travel time $T_{\Cyl(z,s)}(\Hb_0,\Hb_{n\|z\|})$ is
clearly dominated by the minimum of $q$ independent random variables
distributed as $T_{\Cyl(z,r)}(\Hb_0,\Hb_{n\|z\|})$. Thus,
\begin{eqnarray*}
&&\Pr\bigl(T_{\Cyl(z,s)}(\Hb_0,
\Hb_{n\|z\|})-n\mu_{\Cyl
(z,r)}>\varepsilon x\|z\| \bigr)
\\
&&\qquad\le\Pr\bigl(T_{\Cyl(z,r)}(\Hb_0,\Hb_{n\|z\|})-n
\mu_{\Cyl(z,r)}>\varepsilon x\|z\| \bigr)^q, 
\end{eqnarray*}
which by Lemma~\ref{lmaLDtail} is at most $M_2^q/x^q$. Since $\mu
_{\Cyl(z,r)}-\muZ(z)\le\varepsilon\|z\|$ by the choice of $r$, we obtain
%
\begin{equation}
\label{eqtubebound} \Pr\bigl(T_{\Cyl(z,s)}(\Hb_0,
\Hb_{n\|z\|})-n\muZ(z)>2\varepsilon x\|z\| \bigr) \le\frac{M_2^q}{x^q}.
\end{equation}
It remains to connect to the starting and ending points $0$ and $nz$.
Subadditivity gives that
\[
T(0,nz) \le\sum_{v\in V_0(z,s)} \bigl(T(0,v)+T(v+nz,nz)
\bigr)+T_{\Cyl(z,s)}(\Hb_0,\Hb_{n\|z\|}),
\]
so we may via Lemma~\ref{lmaYcond} find an $M_3'$, depending on
$\varepsilon$, $d$ and $s$, such that
\[
\Pr\bigl(T(0,nz)-n\muZ(z)>3\varepsilon x\|z\| \bigr) \le M_3'
\Pr\bigl(Y>x/M_3'\bigr)+\frac{M_2^q}{x^q}.
\]
Since $\varepsilon>0$ was arbitrary, the proof is complete.
\end{pf}

Before completing the proof of Theorem~\ref{teoLDabove}, we will show
that travel times cannot be too large. This result will help us to
loose the dependence on $z$ still present in Proposition~\ref{propLDtail}.

\begin{prop}\label{propLDM}
Assume that $\E[Y^\alpha]<\infty$ for some $\alpha>0$. For every
$q\ge1$ there is a constant $M_4=M_4(\alpha,q,d)$ such that for all
$z\in\Z^d$ and $x\ge\|z\|$
\[
\Pr\bigl(T(0,z)>M_4x \bigr) \le M_4 \Pr(Y>x)+
\frac{1}{x^q}.
\]
\end{prop}

\begin{pf}
Again, assume that $z$ lies in the first orthant. Let $M_3$ denote the
constant figuring in Proposition~\ref{propLDtail} (for given $\alpha
$ and $q$, and with $\varepsilon=1$ and $z=\ebf_1$). The point $z$
can be reached from $0$ in $d$ steps by in each step taking $z_i$ steps
in direction $\ebf_i$, for $i=1,2,\ldots,d$. Thus, due to
subadditivity and Proposition~\ref{propLDtail},
\[
\Pr\bigl(T(0,z)>dM_4'x \bigr)\le\sum
_{i=1}^d\Pr\bigl(T(0,z_i\ebf
_i)>M_4'x \bigr)\le dM_3
\Pr(Y>x/M_3)+\frac{dM_3}{x^q}
\]
for any $M_4'\ge\muZ(\ebf_1)+1$. Hence, $M_4=d^2M_3M_4'$ is sufficient.
\end{pf}

\begin{pf*}{Proof of Theorem~\ref{teoLDabove}}
Fix $\varepsilon>0$. To start, we will choose a finite set of
directions so that each $z\in\Z^d$ will be within distance
$\varepsilon\|z\|$ of some straight line intersecting the origin, and
continues in one of the chosen directions. The set of directions can be
chosen as $\{y\in\Z^d\dvtx \|y\|=N\}$, given that $N$ is large enough.
More precisely, pick $N\ge d/\varepsilon$ and note that any $z\in\Z
^d$ satisfying $mN\le\|z\|<(m+1)N$ will be within $\ell^1$-distance
$N+dm$ of some point in $\{my\dvtx \|y\|=N\}$. In particular, for $\|z\|\ge
N/\varepsilon$ we have
\[
N+dm \le\varepsilon\|z\|+d\|z\|/N \le2\varepsilon\|z\|.
\]

Given $z\in\Z^d$, let $m_z:=\max\{m\ge0\dvtx mN\le\|z\|\}$. First, note
that if $\|z\|\le N/\varepsilon$, then
\[
\Pr\bigl(T(0,z)-\muZ(z)>\varepsilon x \bigr) \le\Pr\bigl
(T(0,z)>\varepsilon x
\bigr) \le9^{2d}(N/\varepsilon)\Pr\bigl(Y>\varepsilon^2
x/(9N) \bigr),
\]
by Lemma~\ref{lmaYcond}. So, we may proceed assuming that $\|z\|\ge
N/\varepsilon$. For any $y$ with $\|y\|=N$, we have
%
\[
T(0,z)-\muZ(z) \le T(0,m_zy)-m_z\muZ(y)+T(m_zy,z)+
\bigl(m_z\muZ(y)-\muZ(z) \bigr),
\]
and, for at least one of these $y$ (the one closest to $z$), since
$\muZ$ satisfies the properties of a norm,
%
\begin{equation}
\label{eqTmzmu} m_z\muZ(y)-\muZ(z) \le\muZ(m_zy-z) \le
\muZ(\ebf_1)\| m_zy-z\| \le2\muZ(\ebf_1)
\varepsilon\|z\|.
\end{equation}
Moreover, with $M_4$ as in Proposition~\ref{propLDM}, and $x\ge\|z\|
\ge\|m_zy-z\|/(2\varepsilon)$,
%
\begin{equation}
\label{eqTmz2} \Pr\bigl(T(m_zy,z)>2M_4\varepsilon x
\bigr) \le M_4\Pr(Y>2\varepsilon x)+\frac{1}{(2\varepsilon x)^q}.
\end{equation}

Since $N$ depends on nothing but $\varepsilon$, there is a constant
$M_3=M_3(\varepsilon,\alpha,q,d)$, given by Proposition~\ref
{propLDtail}, such that for every $y$ satisfying $\|y\|=N$, and $x\ge
\|z\|\ge m_zN$, then
%
\begin{equation}
\label{eqTmz3} \Pr\bigl(T(0,m_zy)-m_z\muZ(y)>
\varepsilon x \bigr) \le M_3\Pr\bigl(Y>x/(M_3N) \bigr)+
\frac{M_3N^q}{x^q}.
\end{equation}
Combining~(\ref{eqTmzmu}),~(\ref{eqTmz2}) and~(\ref{eqTmz3}), we
conclude that also for $x\ge\|z\|\ge N/\varepsilon$,
\[
\Pr\bigl(T(0,z)-\muZ(z)>\bigl(1+2\muZ(\ebf_1)+2M_4
\bigr)\varepsilon x \bigr) \le M \Pr(Y>x/M)+\frac{M}{x^q},
\]
where $M$ can be taken as the maximum of $M_4+M_3$, $1/(2\varepsilon
)^q$, and $M_3N^q$. Since $\varepsilon>0$ was arbitrary, this
completes the proof.
\end{pf*}

\section{Proof of the Hsu--Robbins--Erd{\H o}s strong law}\label{secHREproof}

Both Theorem~\ref{teoHRE} and Corollary~\ref{corLDsum} may be
thought of as strong laws of the kind introduced by Hsu, Robbins and
Erd\H{o}s. They are easily derived in a similar fashion from the large
deviation estimates presented in Theorems~\ref{teoLDbelow} and~\ref
{teoLDabove}. For that reason, we only present a proof of the former.

\begin{pf*}{Proof of Theorem~\ref{teoHRE}}
Deviations below and above the time constant are easily handled
separately via the identity
\begin{eqnarray*}
\Pr\bigl(\bigl|T(0,z)-\muZ(z)\bigr|>\varepsilon\|z\| \bigr)
&=& \Pr\bigl(T(0,z)-\muZ(z)<-\varepsilon\|z\| \bigr)
\\
&&{} +\Pr\bigl(T(0,z)-\muZ(z)>\varepsilon\|z\| \bigr).
\end{eqnarray*}
Summability of the probabilities of deviations below the time constant
is immediate from Theorem~\ref{teoLDbelow}, since $\Pr
(T(0,z)-\muZ(z)<-\varepsilon\|z\| )$ decays exponentially in $\|
z\|$, while the number of sites satisfying $\|z\|=n$ grows polynomially
in $n$.

Consider instead deviations above the time constant. Assume first that
$\E[Y^\alpha]<\infty$ for some $\alpha>0$. According to
Theorem~\ref{teoLDabove}, with $q=\alpha+1$, there is a constant
$M=M(\alpha,\varepsilon,d)$ such that
\begin{eqnarray*}
&&\sum_{z\in\Z^d}\|z
\|^{\alpha-d} \Pr\bigl(T(0,z)-\muZ(z)>\varepsilon\|z\| \bigr)
\\
&&\qquad\le M\sum_{z\in\Z^d} \biggl(\|z\|^{\alpha-d}
\Pr(Y>\|z\| /M)+\frac{1}{\|z\|^{d+1}} \biggr).
\end{eqnarray*}
Observe that the number of $z\in\Z^d$ for which $\|z\|=n$ is of order
$n^{d-1}$. The above summation is therefore finite since $\E[Y^\alpha
]<\infty$ implies that
\[
\sum_{n=1}^\infty n^{\alpha-1}
\Pr(Y>n/M)+\sum_{n=1}^\infty n^{-2}<
\infty.
\]

For the necessity of $\E[Y^\alpha]$ being finite, note that $T(0,z)$
is at least as large as the minimum value among the $2d$ edges adjacent
to $z$. For all large enough $M$, we therefore have
\begin{eqnarray*} \sum_{z\in\Z^d}\|z
\|^{\alpha-d} \Pr\bigl(T(0,z)-\muZ(z)>\varepsilon\|z\| \bigr) &\ge&
\sum
_{z\in\Z^d}\|z\|^{\alpha
-d} \Pr\bigl(Y>M\|z\|\bigr)
\\
&\ge&\sum_{n=1}^\infty n^{\alpha-1}
\Pr(Y>Mn),
\end{eqnarray*}
which is finite only if $\E[Y^\alpha]<\infty$.
\end{pf*}

\section{The sets of points in space and time of linear order deviations}\label{sectimediv}

In this section, we study the set of times $t$ for which the random set
of sites reachable within time $t$ from the origin deviates by as much
as a constant factor from the asymptotic shape. In particular, we will
see how Theorem~\ref{teoHRE} and the estimates on large deviation
above and below the time constants can be used to estimate moments of
$|\Tdiv_\varepsilon|$ and prove Theorem~\ref{teoST}. This will be
done by estimating the contribution of each site $z\in\Z^d$ to $\Tdiv
_\varepsilon$ separately.

Let $Y(z)$ denote the minimum of the $2d$ weights associated with the
edges incident to $z$. Note that $Y(y)$ and $Y(z)$ are independent as
soon as $y$ and $z$ are at \mbox{$\ell^1$-}distance at least 2. Since
$T(0,z)$ is at least as large as $Y(z)$, it is possible to obtain a
sufficient condition for $z$ to be contained in $\Zdiv_\varepsilon$
in terms of $Y(z)$. Recall that either $\muZ\equiv0$ or $\muZ$ is
bounded away from 0 and infinity on compact sets not containing the
origin. Consequently, $\bar{\muZ}:=\sup_{\|x\|=1}\muZ(x)$ is
finite. Thus, $Y(z)>(\bar{\muZ}+\varepsilon)\|z\|$ implies
that $z\in\Zdiv_\varepsilon$, and
%
\begin{equation}
\label{eqZbound} |\Zdiv_\varepsilon| = \sum_{z\in\Z^d}
\ind_{\{|T(0,z)-\muZ
(z)|>\varepsilon\|z\|\}} \ge\sum_{z\in\Z^d}
\ind_{\{Y(z)>\beta
\|z\|\}}, 
\end{equation}
for large enough $\beta=\beta(\varepsilon)$.\vadjust{\goodbreak}

A similar estimate can be obtained for the Lebesgue measure of $\Tdiv
_\varepsilon$ as well. Assume until this end that $\muZ\not\equiv
0$, in which case $\underline{\muZ}:=\inf_{\|x\|=1}\muZ(x)$ is
strictly positive. For $t\ge0$, let
\begin{eqnarray*}
A_t&:=&\bigl\{z\in\Z^d\dvtx T(0,z)>t\mbox{ and } \muZ(z)\le
t(1-\varepsilon)\bigr\},
\\
B_t&:=&\bigl\{z\in\Z^d\dvtx T(0,z)\le t\mbox{ and }
\muZ(z)>t(1+\varepsilon)\bigr\}.
\end{eqnarray*}
Note that $A_t\neq\varnothing$ is equivalent to $\Bc^{\muZ
}_{(1-\varepsilon)t}\not\subset\Bc_t$. Similarly, $B_t\neq
\varnothing$ if and only if $\Bc_t\not\subset\Bc^{\muZ
}_{(1+\varepsilon)t}$. Thus, $\Tdiv_\varepsilon=\{t\ge0\dvtx A_t\cup
B_t\neq\varnothing\}$, and the contribution of a site $z$ is given by
the interval of time for which $z$ is contained in either $A_t$ or
$B_t$. Denote these intervals by $I_A(z)$ and $I_B(z)$, respectively,
and note that
\[
\Tdiv_\varepsilon=\bigcup_{z\in\Z^d}I_A(z)
\cup I_B(z).
\]
Crude but useful upper bounds on the length of $I_A(z)$ and $I_B(z)$
are given by $T(0,z)$ and $\muZ(z)/(1+\varepsilon)$, respectively.
More precisely, we have
%
\begin{eqnarray}\label{eqTupperA}
\bigl|I_A(z)\bigr| &=& \bigl(T(0,z)-
\muZ(z)/(1-\varepsilon) \bigr)\ind_{\{
I_A(z)\neq\varnothing\}}
\nonumber\\[-8pt]\\[-8pt]
&\le&\bigl(T(0,z)-\muZ(z) \bigr)\ind_{\{T(0,z)-\muZ(z)>\beta\|z\|
\}}\nonumber
\end{eqnarray}
and
%
\begin{eqnarray}\label{eqTupperB} %
\bigl|I_B(z)\bigr| &=& \bigl(
\muZ(z)/(1+\varepsilon)-T(0,z) \bigr)\ind_{\{
I_B(z)\neq\varnothing\}}
\nonumber\\[-8pt]\\[-8pt]
&\le&\muZ(z)\ind_{\{T(0,z)-\muZ(z)<-\beta\|z\|\}},\nonumber
\end{eqnarray}
for any $\beta>0$ not larger than $\varepsilon\underline{\muZ
}/(1+\varepsilon)$. Moreover, for $\beta>\bar{\muZ
}/(1-\varepsilon)$
%
\begin{equation}
\label{eqTlower} \bigl|I_A(z)\bigr| \ge\bigl(Y(z)-\beta\|z\| \bigr)
\ind_{\{Y(z)>\beta\|z\|\}}.
\end{equation}
A first consequence of these representations for $|\Zdiv_\varepsilon
|$ and $|\Tdiv_\varepsilon|$ is the following simple observation.

\begin{prop}\label{propYd}
If $\E[Y^d]=\infty$, then $|\Zdiv_\varepsilon|$ and $|\Tdiv
_\varepsilon|$ are infinite for all $\varepsilon>0$, almost surely.
\end{prop}

\begin{pf}
Recall that the $Y(z)$'s are independent for points at $\ell
^1$-distance at least~2. $\E[Y^d]=\infty$ implies that $\sum_{\|z\|
\in2\N}\Pr(Y(z)>\beta\|z\|)=\infty$ for every \mbox{$\beta>0$}.
Consequently, $Y(z)>\beta\|z\|$ for infinitely many $z\in\Z^d$ via
the Borel--Cantelli lemma, almost surely, so $|\Zdiv_\varepsilon|$
is almost surely infinite. The same argument shows that also
$Y(z)>\beta\|z\|+1$ for infinitely many $z\in\Z^d$, and hence
(assuming $\muZ\not\equiv0$) $|\Tdiv_\varepsilon|$ is infinite
almost surely.
\end{pf}

On the other hand, that $\E[Y^d]<\infty$ is sufficient for the
expected cardinality of $\Zdiv_\varepsilon$ to be finite is
immediate from Theorem~\ref{teoHRE}. The first hint to why \mbox{$\E
[Y^{d+1}]<\infty$} is required in order for the expected Lebesgue
measure of $\Tdiv_\varepsilon$ to be finite is that although the
cardinality of $\Zdiv_\varepsilon$ is finite, the furthest point may
lie very far from the origin. For the furthest point to be expected
within finite distance, it is necessary that $\E[Y^{d+1}]<\infty$.
With a slight abuse of notation, we let $\sup\Zdiv_\varepsilon$
denote the \mbox{$\ell^1$-}distance to the furthest point in $\Zdiv
_\varepsilon$.

\begin{prop}\label{propsupZ}
For every $\alpha>0$ and $\varepsilon>0$,
\[
\E\bigl[Y^{d+\alpha}\bigr]<\infty\quad\Longleftrightarrow\quad\E\bigl[(\sup
\Zdiv
_\varepsilon)^\alpha\bigr]<\infty.
\]
\end{prop}

\begin{pf}
Fix $\alpha>0$ and $\varepsilon>0$. The sufficiency of $\E
[Y^{d+\alpha}]<\infty$ is immediate from Theorem~\ref{teoHRE}, since
\[
\E\bigl[(\sup\Zdiv_\varepsilon)^\alpha\bigr]\le\sum
_{z\in\Z^d}\|z\| ^\alpha\Pr\bigl(\bigl|T(0,z)-\muZ(z)\bigr|>
\varepsilon\|z\| \bigr).
\]

It remains to show that $\E[Y^{d+\alpha}]<\infty$ is also necessary.
That $\E[Y^d]<\infty$ is necessary is a consequence of
Proposition~\ref{propYd}, so there is no restriction assuming that
$\E[Y^d]$ is finite. In order to obtain a lower bound on $\sup\Zdiv
_\varepsilon$, we are, in contrast to the lower bound in~(\ref
{eqZbound}), looking for the largest integer $n$ such that there
exists $z\in\Z^d$ for which $\|z\|=n$ and $Y(z)>(\bar{\muZ
}+\varepsilon)\|z\|$. Since $Y(z)$'s are independent for points at
$\ell^1$-distance at least 2, we restrict focus further to even values
of $n$.

For $\beta>0$, let
\[
\eta_\beta:= \bigl\{n\in2\N\dvtx \exists z\in\Z^d\mbox{ for
which }\| z\|=n\mbox{ and }Y(z)>\beta n \bigr\}.
\]
For $\beta=\beta(\varepsilon)$ sufficiently large ($\beta\ge
\bar{\muZ}+\varepsilon$ will do), we have the lower bound
%
\begin{eqnarray}\label{eqsupZbound}
\E\bigl[(\sup\Zdiv_\varepsilon
)^\alpha\bigr] &\ge&\sum_{n\in2\N
}n^\alpha
\Pr(\sup\eta_\beta=n)
\nonumber\\[-8pt]\\[-8pt]
&=& \sum_{n\in2\N}n^\alpha\Pr\Bigl(\max
_{\|z\|=n}Y(z)>\beta n \Bigr)\Pr(\sup\eta_\beta\le n),\nonumber
\end{eqnarray}
where the equality follows by independence.

It is well known that the probability of a binomially distributed
random variable being strictly positive is comparable to its mean. Let
$X$ be binomially distributed with parameters $n$ and $p$. The union
bound shows that its mean $np$ is an upper bound on $\Pr(X>0)$, but an
application of Cauchy--Schwarz's inequality gives as well the lower
bound $\E[X]^2/\E[X^2]\ge np/(1+np)$, which if $np$ is small compared
to 1, is at least $np/2$.

Fix $\beta=\beta(\varepsilon)$ such that~(\ref{eqsupZbound})
holds. Let $X_n$ denote the number of $z$ for which $\|z\|=n$ and
$Y(z)>\beta n$. Since $Y(z)$'s are independent for different points at
the same $\ell^1$-distance, $X_n$ is binomial. The number of points at
distance $n$ from the origin are of order $n^{d-1}$, and (since $\E
[Y^d]<\infty$ is assumed) $\Pr(Y>\beta n)$ decays at least as
$n^{-d}$ via Markov's inequality. Consequently, for $n$ large $\E
[X_n]\le1$, and there is $\delta>0$ such that
%
\begin{equation}
\label{eqbintrick} \Pr\Bigl(\max_{\|z\|=n}Y(z)>\beta n \Bigr) \ge
\E[X_n]/2 \ge\delta n^{d-1} \Pr(Y>\beta n).
\end{equation}
We next claim that $\Pr(\sup\eta_\beta\le n)\ge1/2$ for large $n$.
Via the lower bound in~(\ref{eqsupZbound}), we conclude that for some
$\delta>0$ and $N<\infty$
\[
\E\bigl[(\sup\Zdiv_\varepsilon)^\alpha\bigr] \ge\frac{\delta}{2}
\sum_{n\in2\N\dvtx  n\ge N}n^{d+\alpha-1} \Pr(Y>\beta n),
\]
for which $\E[Y^{d+\alpha}]<\infty$ is necessary in order to be finite.

It remains to show that $\Pr(\sup\eta_\beta\le n)\ge\frac{1}{2}$
for large enough $n$. Note that
\[
\Pr(\sup\eta_\beta>n) \le\sum_{k>n}\Pr
\Bigl(\max_{\|z\|
=k}Y(z)>\beta k \Bigr) \le C\sum
_{k>n}k^{d-1} \Pr(Y>\beta k),
\]
for some $C<\infty$. Since $\E[Y^d]$ is assumed finite, the
right-hand side is\break summable, and becomes arbitrarily small as $n$
increases. This proves the claim.
\end{pf}

What remains is to prove Theorem~\ref{teoST}. The proof will be
similar to that of Proposition~\ref{propsupZ}, but will require a
couple of additional estimates. The difference is a consequence of the
difference in the upper and lower bounds between~(\ref{eqZbound})
and \mbox{(\ref{eqTupperA})--(\ref{eqTlower})}. As such, we will encounter
moments of products on the form $X\cdot\ind_{\{X>a\}}$. For every
$\alpha>0$, $a\ge0$ and random variable $X$, we have the following formula:
%
\begin{equation}
\label{eqexprestriction} \E\bigl[X^\alpha\cdot\ind_{\{X>a\}}\bigr] =
a^\alpha\Pr(X>a)+\alpha\int_a^\infty
x^{\alpha-1} \Pr(X>x) \,dx.
\end{equation}
Since it will be used more than once, we separate the following bound
on a double summation as a lemma.

\begin{lma}\label{lmadoublesum}
For every $\alpha\ge0$ and $\beta\ge0$, there are $c=c(\alpha,\beta)$
and $C=C(\alpha,\beta)$ such that
\begin{eqnarray*} c\sum_{n=1}^\infty(n-1)^{\alpha+\beta+1}
\Pr(X>n) &\le&\sum_{n=1}^\infty
n^\alpha\int_n^\infty x^\beta
\Pr(X>x) \,dx
\\
&\le& C\sum_{n=1}^\infty(n+1)^{\alpha+\beta+1}
\Pr(X>n).
\end{eqnarray*}
\end{lma}

\begin{pf}
Split the integration domain into unit intervals and bound the
integrand from below and above. The lower and upper bonds follow via
the estimate
\[
(m/2)^{\alpha+1} \le\sum_{n=1}^mn^\alpha
\le m^{\alpha+1}.
\]\upqed
\end{pf}

\begin{pf*}{Proof of Theorem~\ref{teoST}}
Fix $\alpha>0$ and $\varepsilon>0$. We will prove the implications,
one by one, between the three expressions: (a) $\E[Y^{d+\alpha
}]<\infty$; (b) $\E[(\sup\Tdiv_\varepsilon)^\alpha
]<\infty$; and (c) $\E[|\Tdiv_\varepsilon|^\alpha]<\infty
$, starting with the following:\vspace*{6pt}

(a)${}\Rightarrow{}$(b)
Since $\sup\Tdiv_\varepsilon=\sup\{\max I_A(z)\cup I_B(z)\dvtx z\in\Z
^d\}$,
we obtain in similarity to~(\ref{eqTupperA}) and~(\ref{eqTupperB})
the upper bound
%
\begin{eqnarray*} \E\bigl[(\sup\Tdiv_\varepsilon)^\alpha
\bigr] &\le&\sum_{z\in
\Z^d}\E\bigl[ \bigl(\max
I_A(z)\cup I_B(z) \bigr)^\alpha\bigr]
\\
&\le&\sum_{z\in\Z^d}\E\bigl[T(0,z)^\alpha
\ind_{\{T(0,z)-\muZ
(z)>\beta\|z\|\}} \bigr]
\\
&&{} +\sum_{z\in\Z^d}\muZ(z)^\alpha\Pr
\bigl(T(0,z)-\muZ(z)<-\beta\|z\| \bigr),
\end{eqnarray*}
valid for all sufficiently small $\beta=\beta(\varepsilon)>0$. The
latter sum in the right-hand side is finite as of Theorem~\ref
{teoHRE}. The former sum takes via~(\ref{eqexprestriction}) the form
\begin{eqnarray*}
&& \sum_{z\in\Z^d}\big(\muZ(z)+\beta\|z\|\big)^\alpha\,\Pr\big(T(0,z)-\muZ(z)>\beta\|z\|\big)
\\
&&\quad{} +\sum_{z\in\Z^d}\alpha\int_{\muZ(z)+\beta\|z\|}^\infty x^{\alpha-1}\,\Pr\big(T(0,z)>x\big)\,dx.
\end{eqnarray*}
The former of these two sums is again finite according to Theorem~\ref
{teoHRE}. Once the latter sum is broken up into two, one over $n\in\N
$ and the other over $\|z\|=n$, and the integral has gone through a
chance of variables $x\mapsto x-\muZ(z)$, then Theorem~\ref
{teoLDabove} can be used to relate the probability tail of
$T(0,z)-\muZ(z)$ with that of $Y$. Since the number of points at
distance $n$ from the origin is of order $n^{d-1}$, an upper bound is
given by
\[
\alpha CM\sum_{n=1}^\infty n^{d-1}
\int_{\beta n}^\infty x^{\alpha
-1} \biggl(\Pr(Y>x/M)+
\frac{1}{x^{d+\alpha+1}} \biggr) \,dx,
\]
for some finite constants $C$ and $M=M(\alpha,\beta,d)$.
Integrating over the two terms separately breaks the sum in two, of
which the latter is easily seen to be finite. The former can instead be
estimated via Lemma~\ref{lmadoublesum}. An upper bound on this part
is obtained as
\[
M'\sum_{n=1}^\infty(n+1)^{d+\alpha-1}
\Pr(Y>\beta n/M),\vadjust{\goodbreak}
\]
for some constant $M'=M'(\alpha,\beta,d)$, which is finite when $\E
[Y^{d+\alpha}]<\infty$.

(b)${}\Rightarrow{}$(c) This step is trivial.

(c)${}\Rightarrow{}$(a)
A lower bound on $|\Tdiv_\varepsilon|$ is given by the contribution
of a single site $z$. However, looking at a particular site is not
going to give a bound of the right order. Instead we will pick a site
randomly, and more precisely, the site furthest from the origin among
those contributing to $\Tdiv_\varepsilon$. In case this site is not
unique, then we pick the one contributing more. The contribution of
each site $z$ was in~(\ref{eqTlower}) seen to be at least
$(Y(z)-\beta\|z\|)\ind_{\{Y(z)>\beta\|z\|\}}$, for every
sufficiently large $\beta=\beta(\varepsilon)$. Similar to that
of~(\ref{eqsupZbound}), we have
\[
\E\bigl[|\Tdiv_\varepsilon|^\alpha\bigr] \ge\sum
_{n\in2\N}\E\Bigl[ \Bigl(\max_{\|z\|=n}Y(z)-\beta
n \Bigr)^\alpha\ind_{\{\max_{\|z\|
=n}Y(z)>\beta n\}} \Bigr]\Pr(\sup\eta_\beta\le
n).
\]
Here, like in the proof of Proposition~\ref{propsupZ}, we sum over
even integers for the sake of independence between the events $\{\sup
\eta_\beta=n\}$ and $\{\sup\eta_\beta\le n\}$. Combining the
identity~(\ref{eqexprestriction}) and the bound~(\ref{eqbintrick})
with a change of variables, we arrive at
\[
\E\bigl[|\Tdiv_\varepsilon|^\alpha\bigr] \ge\alpha\delta\sum
_{n\in2\N}n^{d-1} \Pr(\sup\eta_\beta\le
n)\int_{\beta n}^\infty(x-\beta n)^{\alpha-1} \Pr
\bigl(Y(z)>x\bigr) \,dx.
\]
Again, since $\Pr(\sup\eta_\beta\le n)\ge1/2$ for $n$ large
enough, we may via the double summation estimate in Lemma~\ref
{lmadoublesum} conclude that $\E[Y^{d+\alpha}]<\infty$ is necessary
for $\E[|\Tdiv_\varepsilon|^\alpha]$ to be finite.
\end{pf*}

\begin{appendix}
\section{Convergence toward the time constant}\label{appmu}

The time constant was in~(\ref{deftimeconstant}) defined for $z\in\Z
^d$ as the limit in probability of $\frac{1}{n}T(0,nz)$ as $n\to
\infty$. Existence of the limit (almost surely and in $L^1$) under the
assumption $\E[Y]<\infty$ follows from a straightforward application
of the subadditive ergodic theorem \cite{kingman68}. Existence of the
limit (in probability) without a moment condition was later derived
in~\cite{coxdur81,kesten86}. There is a unique extension of $\muZ$ to
all of $\R^d$ that retains the properties of a semi-norm. For
example, we may define $\muZ(x)$ for $x\in\R^d$ via the limit
\[
\muZ(x):=\lim_{n\to\infty}\frac{\muZ(v_n)}{n},
\]
where $v_1,v_2,\ldots$ is any sequence in $\Z^d$ such that $v_n/n\to
x$ as $n\to\infty$.

Existence of this limit is well known and follows from the properties
of $\mu$ as a semi-norm on $\Z^d$. That these properties are
preserved in the limit is similarly verified. We would here like to
emphasize, perhaps, a less-known fact, which is easily seen to follow
the results reported in this paper. Namely, the necessary and
sufficient condition under which $\muZ(x)$, for $x\in\R^d$, appears
as the almost sure limit for some sequence of travel times. We are not
aware of such a condition previously appearing in the literature.

\begin{prop}
Fix $x\in\R^d$ and let $v_1,v_2,\ldots$ be any sequence of points in
$\Z^d$ such that $v_n/n\to x$ as $n\to\infty$. Then
\[
\lim_{n\to\infty}\frac{T(0,v_n)}{n}=\muZ(x)\qquad\mbox{in
probability}.
\]
Moreover, the limit holds almost surely, in $L^1$ and completely if and
only if \mbox{$\E[Y]<\infty$.}
\end{prop}

\begin{pf}
A first observation due to the triangle inequality is that
\[
\bigl|T(0,v_n)-n\muZ(x)\bigr| \le\bigl|T(0,v_n)-
\muZ(v_n)\bigr|+\bigl|\muZ(v_n)-n\muZ(x)\bigr|.
\]
Let $\varepsilon>0$. By the properties of $\muZ$ as a norm, the
latter term in the right-hand side is bounded above by $\muZ(\ebf_1)\|
v_n-nx\|$, which is at most $\varepsilon n$ when $n$ is large. It
follows that
\[
\limsup_{n\to\infty}\Pr\bigl(\bigl|T(0,v_n)-n\muZ(x)\bigr|>2
\varepsilon n \bigr) \le\limsup_{n\to\infty}\Pr\bigl(\bigl|T(0,v_n)-
\muZ(v_n)\bigr|>\varepsilon n \bigr),
\]
which by~(\ref{eqweakST}) has to equal zero. This proves convergence
in probability.

Necessity of $\E[Y]<\infty$ for almost sure and $L^1$-convergence
follows as before, due to the fact that a lower bound on the travel
time from the origin to any other point is bounded from below by the
minimum of the $2d$ weights associated with the edges adjacent to the origin.
To conclude that $\E[Y]<\infty$ is sufficient for the convergence to
hold almost surely it suffices to note that the sequence $(nz)_{n\ge
1}$ in Corollary~\ref{corLDsum} can be exchanged for any sequence
$(v_n)_{n\ge1}$ for which $\|v_n\|/n$ is bounded away from $0$ and
$\infty$. In particular, by Theorems~\ref{teoLDbelow} and~\ref
{teoLDabove} (with $\alpha=1$ and $q=2$) there exists $M$ (depending
on $\varepsilon>0$ and the upper bound on $\|v_n\|/n$) such that
\[
\sum_{n=1}^\infty\Pr\bigl(\bigl|T(0,v_n)-
\muZ(v_n)\bigr|>\varepsilon n \bigr) \le M \sum
_{n=1}^\infty\biggl(\Pr(Y>n/M)+\frac{1}{n^2}
\biggr),
\]
which is finite since $\E[Y]<\infty$. This proves almost sure and
complete convergence.
Finally, $L^1$-convergence is due to Corollary~\ref{corLp}, again
since $\|v_n\|/n$ is assumed bounded.
\end{pf}

\section{Kesten's construction of shells}\label{appshell}

For completeness, let us present a precise construction of the shells
in Section~\ref{secshell}. The construction will follow closely that
of Kesten~\cite{kesten86}, Section~2.

Given $\delta>0$, pick $\tbar=\tbar(\delta)$ such that $\Pr(\tau
_e\le\tbar)\ge1-\delta$. As before, color each vertex in $\Z^d$
either black or white; black if at least one of the edges adjacent to
it has weight larger than $\tbar$, and white otherwise. We shall below
introduce a notion of black and white clusters, for which we will need
to distinguish paths from \mbox{$\star$-}paths. A \emph{path} will refer to
a sequence of vertices $v_0,v_1,\ldots,v_n$ of the $\Z^d$ lattice
such that two consecutive points are at $\ell^1$-distance one. A \emph
{$\star$-path} will similarly refer to a nearest-neighbor sequence of
vertices with respect to $\ell^\infty$-distance on $\Z^d$. A~path or
a $\star$-path will be called \emph{black} or \emph{white} if all
its points are black or white, respectively.

Given $A\subset\Z^d$, define the black and white clusters of $A$ as
\begin{eqnarray*} C(A,b)&:=&A\cup\bigl\{z\in\Z^d\dvtx z
\leftrightarrow y\mbox{ by a black $\star$-path},
\\
&&\phantom{A\cup\bigl\{}\mbox{for some $y$ at $\ell^\infty$-distance 1 from }A\bigr\},
\\
C(A,w)&:=&A\cup\bigl\{z\in\Z^d\dvtx z\leftrightarrow y\mbox{ by a white
path},
\\
&&\phantom{A\cup\bigl\{} \mbox{for some $y$ at $\ell^1$-distance 1 from }A\bigr\}.
\end{eqnarray*}
The exterior boundary $\partial_{\mathrm{ext}}C$ of a set $C\subset
\Z^d$ is defined as the set of points $z\in\Z^d\setminus C$ for
which there is a point $y\in C$ at $\ell^\infty$-distance 1 from $z$,
and for which there is a path connecting $z$ to infinity without
intersecting $C$. Next, let $D_n(z)$ denote the box of side-length
$2n+1$ centered at $z$, and let
\[
n(z):=\min\bigl\{n\ge0\dvtx \bigl|C(y,w)\bigr|=\infty\mbox{ for some }y\in D_n(z)
\bigr\}.
\]

Let $S_z:=\partial_{\mathrm{ext}}C(D_{n(z)}(z),b)$. By construction,
all vertices in $S_z$ are white, and $S_z$ must contain a white vertex
belonging to an infinite white cluster. Kesten continues with two lemmas.

\begin{lma}[(\cite{kesten86}, Lemma~2.23)]\label{lmakesten1}
The exterior boundary of any finite \mbox{$\star$-}connec\-ted set is
connected. In particular, if $C(D_{n(z)}(z),b)$ is finite, then $S_z$
is connected. Moreover, in that case $S_z$ separates $z$ from infinity
in the sense that every path from $z$ to infinity has to intersect $S_z$.
\end{lma}

\begin{lma}[(\cite{kesten86}, Lemma~2.24)]
If $\delta>0$ is sufficiently small, then the set $C(D_n(z),b)$ is
almost surely finite for every $n\ge0$. Moreover, there are constants
$M<\infty$ and $\gamma>0$ such that for every $k\ge0$
\[
\Pr\bigl(n(z)>k \bigr) \le Me^{-\gamma k}\quad\mbox{and}\quad\Pr\bigl(
\diam(S_z)>k \bigr) \le Me^{-\gamma k}.
\]
\end{lma}

The first lemma distinguishes an ``inside'' of $S_z$, consisting of
points separated from infinity by $S_z$. Finally, let $\Delta_z$
denote the union of $S_z$ and each point on the inside of $S_z$ which
is white and connected to $S_z$ by a white path.

Since each point in $S_z$ is white, it follows that $\Delta_z$ is all
white. That $\Delta_z$ almost surely satisfies the first three
properties stated in Section~\ref{secshell} is a consequence of the
above two lemmas. The final property follows from the following lemma.
The lemma is a slight variant of~(2.30) in~\cite{kesten86}, and is proved
similarly. For completeness, we present an argument also here.

\begin{lma}
Either every path between $y$ and $z$ in $\Z^d$ intersects both
$\Delta_y$ and~$\Delta_z$, or $\Delta_y\cap\Delta_z\neq\varnothing$.
\end{lma}

\begin{pf}
Assume that there is a path $\gamma$ connecting $y$ and $z$ which does
not intersect $\Delta_y$. We will show that this implies that $\Delta
_y\cap\Delta_z\neq\varnothing$. Let $\Gamma$ be a path from $z$ to
infinity. $\Gamma$ must intersect both $S_y$ and $S_z$, since there
would otherwise be a path from $y$ to infinity (the concatenation of
$\gamma$ and $\Gamma$) not intersecting $S_y$ or a path from $z$ to
infinity ($\Gamma$) not intersecting $S_z$, thus contradicting
Lemma~\ref{lmakesten1}. Let $v$ denote the first point in $S_y\cup
S_z$ visited by $\Gamma$. We claim that either $v\in S_y\cap S_z$, or
$v$ is contained in just one of them, but connected to the other by a
white path. The claim, if true, would imply that $v\in\Delta_y\cap
\Delta_z$, since $v$ would either be contained in~$S_y$, or contained
in its interior but connected to $S_y$ by a white path.

It remains to prove the claim. Assume that $v\in S_z$. We will next
construct a white path from $v$ to $S_y$. Since $D_{n(z)}(z)$ contains
a vertex in an infinite white cluster, there has to be a vertex $u\in
S_z$ that is connected to infinity by a white path. A white path
connecting $v$ to infinity may now be obtained by first connecting $v$
to $u$ within~$S_z$, and thereafter connect $u$ to infinity. This path
will necessarily intersect $S_y$, since we otherwise would have found a
path from $y$ to infinity avoiding $S_y$. The remaining case when $v\in
S_y$ is analogous.
\end{pf}
\end{appendix}

\section*{Acknowledgement}
The author is grateful to Simon Griffiths for a helpful discussion.



%

\printaddresses

\end{document}